\documentclass[12pt, reqno]{article}

\usepackage{fullpage}
\usepackage[usenames]{color}
\usepackage{array}
\usepackage[normalem]{ulem}
\useunder{\uline}{\ul}{}
\usepackage{algorithm}
\usepackage{algpseudocode}
\usepackage{graphicx}
\usepackage{tikz} 
\usetikzlibrary{arrows}
\usepackage{graphics,amsmath,amssymb}
\usepackage{epstopdf}

\usepackage[colorlinks=true,
linkcolor=webgreen,
filecolor=webbrown,
citecolor=webgreen]{hyperref}

\definecolor{webgreen}{rgb}{0,.5,0}
\definecolor{webbrown}{rgb}{.6,0,0}

\usepackage{amsthm} 
\theoremstyle{plain}
\newtheorem{theorem}{Theorem}
\newtheorem{corollary}[theorem]{Corollary}
\newtheorem{res}[theorem]{Result}
\newtheorem{conj}[theorem]{Conjecture}
\newtheorem{lemma}[theorem]{Lemma}
\newtheorem{proposition}[theorem]{Proposition}

\theoremstyle{definition}
\newtheorem{definition}[theorem]{Definition}
\newtheorem{example}[theorem]{Example}

\theoremstyle{remark}
\newtheorem{remark}[theorem]{Remark}

\newcommand{\Al}{\mathcal{A}}

\newcommand{\Cech}{\v{C}ech{} }

\newcommand{\seqnum}[1]{\href{http://oeis.org/#1}{\underline{#1}}}

\begin{document}

\begin{center}
\vskip 1cm
{\LARGE\bf Computations for Symbolic Substitutions}\\
\vskip 1cm
\large
Scott Balchin\\
Department of Mathematics \\ 
University of Leicester \\ 
University Road, Leicester, LE1 7RH \\ 
United Kingdom\\
\href{mailto:slb85@le.ac.uk}{\tt slb85@le.ac.uk} \\
\ \\
Dan Rust\\
Department of Mathematics \\ Universit\"{a}t Bielefeld \\ 
Bielefeld, Universit\"{a}tsstra{\ss}e, D-33615 \\
Germany\\
\href{mailto:drust@math.uni-bielefeld.de}{\tt drust@math.uni-bielefeld.de}\\
\end{center}

\begin{abstract}

We provide a survey of results from symbolic dynamics and algebraic topology relating to Grout, a new user-friendly program developed to calculate combinatorial properties and topological invariants of a large class of symbolic substitutions. We study their subshifts (and related spaces) with an emphasis on examples of computations. We implement a check to verify that no counterexample exists to the so-called ``strong coincidence conjecture'' for a large number of substitutions on three and four letters.
\end{abstract}

\section{Introduction}
	Aperiodic sequences have been well-studied throughout the 20th century, with \emph{substitution systems} being an important subclass in their study since the work of Thue \cite{T:thue-morse}. These are sequences generated from an iterated replacement rule on a finite alphabet that assigns finite words to each letter. Substitutions are also called \emph{morphisms} in the literature. There are several hundred entries containing the term \emph{morphism} in the Online Encyclopedia of Integer Sequences \cite{OEIS}. These sequences are of great interest to the field of combinatorics on words and have been extensively studied \cite{F:book}.
	
	Today, aperiodic order is a topic of general interest to the study of tiling theory, dynamical systems, topology, Diophantine approximation, ergodic theory, computer graphics, mathematical physics and even virology \cite{CH:codim-one-attractors, S:book, BS:beta-numerations, BG:book, L:comp-sci-book, BHZ:gap-labeling, T:virology-tilings}. Once it was discovered by Berger \cite{B:undecidability} that sets of tiles exist which can only tile the plane aperiodically, a flurry of advances quickly followed, culminating in the celebrated discovery of the famous \emph{Penrose tilings} \cite{P:pentaplexity} and in the discovery of quasicrystals \cite{S:nobel} for which Schectman was awarded the Nobel prize in Chemistry in 2011. However, quite apart from the higher-dimensional setting, there is still much that is not well-understood about the low-dimensional case of bi-infinite sequences. In many respects, substitution systems provide the simplest examples of aperiodic sequences.

	The seminal work of Anderson and Putnam \cite{AP} has led to a new algebraic topological approach to the study of aperiodic sequences and tilings associated with substitutions. It is this style of approach, and subsequent related methods for bi-infinite sequences, which we address here. In this setting, 1-dimensional tilings are combinatorially equivalent to bi-infinite sequences over an alphabet whereby letters from the alphabet are assigned to distinct tiles. Informally, the symbolic dynamicist takes the position that (in some instances) a sequence $s$ can be better understood by taking $s$ in concert with all other sequences which are \emph{locally isomorphic} to $s$. One then provides a suitable metric or topological structure to this space of sequences. A simple dynamical system on such a space, given by the \emph{shift map}, is then studied. Combinatorial properties of the original sequence $s$ translate to dynamical properties of this space of sequences. We can in turn translate these dynamical properties of the sequence space into topological properties of a related space---its so-called \emph{tiling space}. Algebraic invariants of tiling spaces are hence important for combinatorially distinguishing aperiodic sequences.
	
	There are still long-standing conjectures concerning the dynamics for symbolic substitutions which need to be settled; principal among them being the \emph{Pisot conjecture} \cite{ABBL:pisot} and related problems. It is hoped that the use of new programs will make testing conjectures in tiling theory and symbolic substitutional dynamics more efficient, as well as allowing for the confirmation of hand calculations and comparing different methods of calculation. In this direction, we present here supporting material for a program (\emph{Grout}) developed to calculate some of these properties (especially methods of calculating cohomology). Analysis of large data sets which can be potentially generated by the Grout source code, and the recognition of underlying patterns in the data may also aid to further the theory.
	
	Grout has been designed with user experience in mind and includes many ease-of-use tools such as the ability to save and load examples, and convenient methods of sharing examples with other users via short strings that encode a substitution. There is also an option to export all of the data that has been calculated to a pre-formatted \LaTeX{} file, including all the TikZ code for the considered complexes.  This is useful to those needing to typeset such diagrams by fully automating the generation of diagrams in TikZ.

In Section \ref{SEC:background}, we introduce basic notions of subshifts (sequence spaces) and tiling spaces associated with substitutions on finite alphabets. In Section \ref{SEC:main}, we introduce the relevant theory from symbolic dynamics. In this section we also include a new fully deterministic check for recognizability for primitive substitutions. The algorithm is simple enough to check by hand in a reasonable amount of time for small substitutions. Section \ref{SEC:cohomology} will cover specifically those methods implemented to compute cohomology for tiling spaces. Throughout, we give instances of these methods being applied to well-known example substitutions. In Section \ref{SEC:pisot} we give a brief overview of results relating to the \emph{Pisot conjecture} and the \emph{strong coincidence conjecture}. We focus our efforts on applying the functions developed for Grout to the study of Pisot substitutions. In particular, we perform a search for counterexamples to the strong coincidence conjecture.

Section \ref{SEC:pisot} is the only section in which original results appear (with the exception of the recognizability check). The preceding sections are meant as a survey of results in order to familiarize the reader with the types of calculations that Grout has been developed to compute.

	\section{Background}\label{SEC:background}
	Most of the definitions that follow can be found in any standard reference text on symbolic dynamics. For one such text, we refer the reader to the book of Fogg \cite{F:book}.
		\begin{definition}
			Let $\mathcal{A}$ be a finite alphabet on $l$ letters (also called symbols). An alphabet $\Al$ may be a set of integers (e.g., $\{0,1,2, \ldots, (l-1)\}$) or a set of letters (e.g., $\{a,b,c\}$). For all positive integers $n$, let $\mathcal{A}^n$ be the $n$-fold cartesian product of the alphabet $\mathcal{A}$ with itself. An \emph{$n$-letter word} is an element of $\mathcal{A}^n$. We formally identify an $n$-letter word, which is an $n$-tuple $(a_1, \ldots, a_n)$ of letters from the alphabet $\mathcal{A}$, with the finite string of letters $a_1 \cdots a_n$. We let $\mathcal{A}^+$ denote the union of the sets of $n$-letter words $\bigcup_{n \geq 1} \mathcal{A}^n$. Further, if the empty word $\epsilon$ is included, we let $\mathcal{A}^\ast$ denote the union $\mathcal{A}^+ \cup \{\epsilon\}$. For $w \in \mathcal{A}^n$, we say that $n$ is the \emph{length} of the word $w$ and write $|w|=n$. The empty word $\epsilon$ has length $|\epsilon| = 0$.
			
			Let $a_1, \ldots, a_n, b_1, \ldots, b_m \in \mathcal{A}$, $w \in \mathcal{A}^*$. \emph{Concatenation} of finite words is a binary operation $\bullet \colon \mathcal{A}^* \times \mathcal{A}^* \to \mathcal{A}^*$ defined by $a_1 \cdots a_n \bullet b_1 \cdots b_m = a_1 \cdots a_n b_1 \cdots b_m$ and $\epsilon \bullet w = w \bullet \epsilon = w$. As concatenation is associative, we will omit this formal notation, and without loss of generality write $(uv)w = u(vw) = uvw$ for words $u,v,w \in \mathcal{A}^*$.
			
			A \emph{substitution} $\phi$ on $\mathcal{A}$ is a function $\phi \colon \mathcal{A} \to \mathcal{A}^+$ which assigns a non-empty word $\phi(a) \in \mathcal{A}^+$ to each letter $a \in \mathcal{A}$. We extend $\phi$ to a function $\phi \colon \mathcal{A}^+ \to \mathcal{A}^+$ by concatenation; given a word $w = w_1 \cdots w_n \in \mathcal{A}^n$, we set $\phi(w) = \phi(w_1) \cdots \phi(w_n)$. In this way, we can consider finite positive powers of the substitution $\phi$ acting on $\mathcal{A}^+$. We define $\phi^1 = \phi$ and $\phi^{p+1} = \phi \circ \phi^{p}$.
			
			An infinite sequence $s$ over the alphabet $\mathcal{A}$ is a function $s \colon \mathbb{N} \to \mathcal{A}$. We write $s$ sequentially as $s = s_0 s_1 s_2 \cdots$.
			A bi-infinite sequence $s$ over the alphabet $\mathcal{A}$ is a function $s \colon \mathbb{Z} \to \mathcal{A}$. We write $s$ sequentially as $s = \cdots s_{-1} \cdot s_0 s_1 \cdots$ where a single dot $\cdot$ is used to mark the position left of the \emph{origin} $s(0) = s_0$.
			
			If $s$ is an infinite sequence $s = s_0 s_1 s_2 \cdots$, or a bi-infinite sequence $s = \cdots s_{-1}\cdot s_0 s_1 \cdots$, then we define the shifted sequence $\sigma(s)$ to be given by $\sigma(s)_n = s_{n+1}$.
		
		We say that a word $w$ is a \emph{factor} of the word $w'$ if there exist words $u$ and $v$ (possibly empty) such that $u w v = w'$. Similarly, the word $w$ of length $|w| = n$ is a factor of the bi-infinite sequence $s$ if there exists an integer $k$, such that $w = s_{k+1} s_{k+2}\cdots s_{k+n}$.
		\end{definition}

		Some prefer to extend the definition of a substitution to maps whose codomain is $\Al^*$, where the empty word $\epsilon$ is allowed to be the image of a letter under substitution. However, such substitutions are difficult to work with in our setting (and a further standing assumption to be introduced will negate this extended definition). So, for our purposes, we shall always require that the image of a letter under a substitution is non-empty.

		\begin{definition}
			Let $\phi \colon \mathcal{A} \to \mathcal{A}^+$ be a substitution. We say a word $w \in \mathcal{A}^\ast$ is \emph{admitted} by the substitution $\phi$ if there exists a letter $i \in \mathcal{A}$ and a natural number $p \geq 0$ such that $w$ is a factor of $\phi^p(i)$. We let $\mathcal{L}^n_\phi \subset \mathcal{A}^n$ denote the set of all words of length $n$ admitted by $\phi$. We form the \emph{language} of $\phi$ by taking the set of all admitted words $\mathcal{L}_\phi = \bigcup_{n \geq 0} \mathcal{L}^n_\phi$.
			
			We say a bi-infinite sequence $s \in \mathcal{A}^\mathbb{Z}$ is \emph{admitted} by $\phi$ if every factor of $s$ is admitted by $\phi$. We let $X_\phi$ denote the set of all bi-infinite sequences admitted by $\phi$. The set $X_\phi$ has a natural (metric) topology inherited from the product topology on $\mathcal{A}^\mathbb{Z}$ and a natural shift map $\sigma \colon X_\phi \to X_\phi$ given by $\sigma(s)_i= s_{i+1}$. We call the pair $(X_\phi, \sigma)$ the \emph{subshift} associated with $\phi$ and we will often abbreviate the pair to just $X_{\phi}$ when the context is clear.
		\end{definition}
		
		An infinite sequence of words $u_0, u_1, \ldots$ is \emph{nested to the right} (or simply \emph{nested} when context is clear) if, for each $n \geq 0$, there exists a word $v_n$ such that $ u_{n+1} = u_n v_n$, where $v_n$ may be the empty word.
		\begin{example}\label{EX:fibonacci}
		Let $\Al = \{0,1\}$ and let $\phi \colon \Al \to \Al^+$ be given by $\phi \colon 0 \mapsto 01, 1 \mapsto 0$. By iterating $\phi$ on the initial seed letter $0$ we get a nested sequence of words:
		$$0 \mapsto 01 \mapsto 010 \mapsto 01001 \mapsto 01001010 \mapsto 0100101001001 \mapsto \cdots, $$
		whose limit sequence is the \emph{Fibonacci word} \seqnum{A003849} \cite{F:book}, (see also \seqnum{A096270}):
		$$010010100100101001010010010100100101001010010010100 \cdots.$$		
		For the purpose of studying the combinatorics and dynamics of such sequences, we do not distinguish this sequence from the version on the alphabet $\{1,2\}$ given by $1\mapsto 12, 2\mapsto 1$ \seqnum{A003842} or on the alphabet $\{a,b\}$ given by $a\mapsto ab, b\mapsto a$, etc.
		\end{example}
		\begin{example}\label{EX:thue-morse}
		Let $\Al = \{0, 1\}$ and let $\phi \colon \Al \to \Al^+$ be given by $\phi \colon 0 \mapsto 01, 1 \mapsto 10$. In this case, the limit sequence given by repeated substitution on the seed letter $0$ is known as a \emph{Thue-Morse sequence} \seqnum{A010060} \cite{P:thue-morse} (see also \seqnum{A010059}, \seqnum{A001285}):
		$$011010011001011010010110011010011001011001101001011 \cdots.$$
		\end{example}
		Note that, in general, a sequence of substituted iterates of a letter does not necessarily form a nested sequence of words extending to the right. However, if the substitution is suitably nice (if there exists an element $a$ of the alphabet, a natural number $n \geq 1$ and a non-empty word $u$ such that $\phi^n(a) = au$), then there exists a finite power of the substitution and some seed letter in the alphabet for which this property does hold. Infinite sequences of letters which are limits of such nested sequences of words whose lengths grow without bound are called \emph{fixed points} of the substitution. Bi-infinite analogues also exist \cite{SW:bi-infinite}, which we study here.
		
		\begin{definition}
Let $\mathcal{A} = \{0, \ldots, (l-1)\}$ be an alphabet with a substitution $\phi \colon \mathcal{A} \to \mathcal{A}^+$. Then $\phi$ has an associated \emph{substitution matrix} $M_\phi$ of dimension $l \times l$ given by setting the entry $m_{ij}$ to be the number of times the letter $i$ appears in the substituted word $\phi(j)$.
\end{definition}
\begin{example}\label{EX:period-doubling}
Let $\Al = \{0,1\}$ and let $\phi \colon \Al \to \Al^+$ be given by $\phi \colon 0 \mapsto 01, 1\mapsto 00$, whose limit sequence with seed $0$ is the \emph{period doubling sequence} \seqnum{A096268} \cite{D:period-doubling}, (see also \seqnum{A035263}):
$$01000101010001000100010101000101 \cdots.$$
This substitution has associated substitution matrix
		$M_\phi =
		\begin{bmatrix}
		1 & 2 \\
		1 & 0
		\end{bmatrix}$.
\end{example}

The first property that we check for a substitution matrix is primitivity. Other properties will be introduced in Section \ref{SEC:matrix}.

\begin{definition}
A substitution $\phi \colon \mathcal{A} \to \mathcal{A}^+$ is called \emph{primitive} if there exists a positive natural number $p$ such that the matrix $M^p_\phi$ has strictly positive entries. Such a matrix $M_\phi$ is also called \emph{primitive}.  This condition is equivalent to having a positive natural number $p$ such that for all $i,j \in \mathcal{A}$ the letter $j$ appears in the word $\phi^p(i)$.
\end{definition}

It is an easy consequence of the definition that any primitive substitution admits a letter $a \in \mathcal{A}$, a natural number $p \geq 1$ and a word $u$ such that $\phi^p(a) = au$. In particular, all primitive substitutions admit a power $p$, and a seed letter $a$, such that the sequence of words $a, \phi^p(a), \phi^{2p}(a), \ldots$ is nested. Note that primitivity is only a sufficient condition for this nesting property to hold, and not necessary.

Checking primitivity of a substitution is algorithmically simple: consider the sequence of matrices $M_\phi, M_\phi^2, M_\phi^3, \ldots$; as the number of zeros appearing as entries in the matrices of this sequence is non-increasing and will remain constant as soon as no difference occurs between steps, either the number of zeros in $M_\phi^k$ and $M_\phi^{k+1}$ are identical and positive for some $k$, in which case $\phi$ is not primitive, or else there exists a power $p$ such that $M_\phi^p$ has all positive entries, in which case $\phi$ is primitive. Exactly one of these events occurs and in finite time.

All examples which have been presented so far are easily checked to be primitive.
\begin{example}\label{EX:chacon}
Let $\Al = \{0,1\}$ and let $\phi \colon \Al \to \Al^+$ be given by $\phi \colon 0 \mapsto 0010, 1 \mapsto 1$, whose limit sequence with seed $0$ is the \emph{non-primitive Chacon sequence} \seqnum{A049320} \cite{C:chacon-sub}:

$$0010 0010 1 0010 0010 0010 1 0010 1 0010 0010 1 0010 \cdots.$$

As the name suggests, this particular substitution that generates the non-primitive Chacon sequence fails to be primitive as its substitution matrix $M_\phi = \left[\begin{smallmatrix}3 & 0 \\ 1 & 1\end{smallmatrix}\right]$ is lower triangular. Hence, $M^p_{\phi}$ has a zero entry for all $p \geq 1$.

Although they are very closely related, the non-primitive Chacon sequence should not be confused with the \emph{primitive Chacon sequence} \seqnum{A049321} \cite{F:chacon}:

$$0012 0012 1 2012 0012 0012 1 2012 1 2012 0012 1 2012 \cdots$$

given as the limit sequence of the substitution $\phi \colon 0 \mapsto 0012, 1 \mapsto 12, 2 \mapsto 012$ with seed $0$. This substitution \emph{is} primitive.
\end{example}

Primitivity is a standing assumption that we ask all considered substitutions to satisfy. For many reasons, non-primitive substitutions are difficult to work with and Grout cannot be used to calculate properties of non-primitive substitutions. The study of non-primitive substitutions via their tiling spaces has only recently been seriously considered \cite{MR:non-primitive}, and we are not aware of any results in the literature concerning the tiling spaces associated with substitutions with letters that have empty image.

\subsection{The subshift and tiling space}
			The subshift associated with a substitution is one of the key objects with which we are concerned. We say $\phi$ is a \emph{periodic} substitution if $X_\phi$ is finite, and $\phi$ is non-periodic otherwise. Equivalently, $\phi$ is periodic if $X_\phi$ contains only periodic sequences, and $\phi$ is non-periodic if $X_\phi$ contains at least one non-periodic sequence (whereas a bi-infinite sequence $s$ is periodic if there exists a non-zero integer $n$ such that $\sigma^n(s)=s$). We say $\phi$ is \emph{aperiodic} if $X_\phi$ contains no periodic shift-invariant subspaces. Equivalently, $\phi$ is aperiodic if $X_\phi$ contains no periodic sequences. We say $\sigma$ is a \emph{minimal} action on $X_\phi$ if the only closed shift-invariant subsets of $X_\phi$ are the empty set $\emptyset$ and the subshift itself $X_\phi$. Equivalently, $\sigma$ is minimal if the orbit of every point under $\sigma$ is dense in $X_\phi$. We say that $X_\phi$ is \emph{minimal} if the shift $\sigma$ is a minimal action on $X_\phi$.
			
			If $\phi$ is non-periodic and primitive, then $\phi$ is automatically aperiodic (this follows from Proposition \ref{PROP:minimal} below). If $\phi$ is aperiodic, then $X_\phi$ is a Cantor set and $\sigma$ is a \emph{minimal} action on $X_\phi$.
			
			If $X_\phi$ is minimal in the sense given above, then one can replace this \emph{internal} definition of a substitution subshift in terms of its language with an \emph{external} definition in terms of the orbit closure of some fixed sequence of the substitution.
			
			Let $A$ be a subset of the metric space $X$. Let $\overline{A}$ denote the closure of $A$. Recall that the closure of $A$ is the intersection of all closed subsets of $X$ which contain $A$. Equivalently, $$\overline{A} = \{ \lim_{n \to \infty} (a_n)_{n \geq 0} \mid (a_n)_{n \geq 0} \in A^\mathbb{N} \}$$ is the collection of all the limit points of $A$.
			\begin{proposition}
			Let $s \in \mathcal{A}^{\mathbb{Z}}$ be a bi-infinite fixed point of the substitution $\phi$ on $\mathcal{A}$. 
			If $X_\phi$ is minimal, then $X_\phi = \overline{\{ \sigma^n(s) \mid n \in \mathbb{Z} \}}$.
			\end{proposition}
			\begin{proof}
			Let $A = \{ \sigma^n(s) \mid n \in \mathbb{Z} \}$. As the fixed point $s$ is clearly admitted by $\phi$, and $X_\phi$ is closed under the action of the shift map $\sigma$, we see that $A$ is a subset of $X_\phi$. It is also clear that $A$ is non-empty and so $\overline{A}$ is also non-empty. Suppose $x \in \overline{A}$. Then there exists a sequence $(a_n)_{n \geq 0}$ of points $a_n \in A$ whose limit is $x$. Note that the shift map is continuous on $X_\phi$ and so $\lim_{n \to \infty} \sigma(a_n) = \sigma (\lim_{n \to \infty} a_n) = \sigma(x)$. It follows that $\sigma(x)$ is in the closure of $A$ because all the points $\sigma(a_n)$ lie in $A$. This means that $\overline{A}$ is a shift-invariant subset of $X_\phi$. As $\overline{A}$ is the intersection of closed sets, it is also closed, and so by shift-invariance and minimality we must have $\overline{A} = X_\phi$.
			\end{proof}
			The following result is well-known \cite{F:book}.
			\begin{proposition}\label{PROP:minimal}
			If $\phi$ is primitive, then $X_\phi$ is minimal.
			\end{proposition}
		
		\begin{definition}
			Let $\phi$ be a substitution on the alphabet $\mathcal{A}$ with associated subshift $X_\phi$. The \emph{tiling space} associated with $\phi$ is the quotient space $$\Omega_\phi = (X_\phi\times [0,1]) /{\sim},$$ where $\sim$ is the equivalence relation generated by the relation $(s,1)\sim (\sigma(s),0)$.
		\end{definition}
		One should imagine the point $(s,t)\in\Omega_\phi$ as being a partition or tiling of the real line $\mathbb{R}$ into labelled unit-length intervals (called tiles), where the labels are determined by the letters appearing in $s$, as shown in Figure \ref{FIG:tiling}, and the origin of $\mathbb{R}$ is shifted a distance $t$ to the right from the left of the tile labelled by the letter $s_0$. Two tilings $T$ and $T'$ in $\Omega_\phi$ are considered $\epsilon$-\emph{close} in this topology if, after a translate by a distance at most $\epsilon$, the tiles around the origin in $T'-\epsilon$ within a ball of radius $1/\epsilon$ lie exactly on top of the tiles around the origin in $T$ within a ball of the same radius and share the same labels.

\begin{figure}[h]
\centering

\begin{tikzpicture}[node distance=0.5cm]
\filldraw [fill=red, draw=black] (0.0,0.0) rectangle (0.5,0.25);
\filldraw [fill=blue, draw=black] (0.5,0.0) rectangle (1.0,0.25);
\filldraw [fill=blue, draw=black] (1.0,0.0) rectangle (1.5,0.25);
\filldraw [fill=red, draw=black] (1.5,0.0) rectangle (2.0,0.25);

\filldraw [fill=blue, draw=black] (2.0,0.0) rectangle (2.5,0.25);
\filldraw [fill=red, draw=black] (2.5,0.0) rectangle (3.0,0.25);
\filldraw [fill=red, draw=black] (3.0,0.0) rectangle (3.5,0.25);
\filldraw [fill=blue, draw=black] (3.5,0.0) rectangle (4.0,0.25);

\filldraw [fill=blue, draw=black] (4.0,0.0) rectangle (4.5,0.25);
\filldraw [fill=red, draw=black] (4.5,0.0) rectangle (5.0,0.25);
\filldraw [fill=red, draw=black] (5.0,0.0) rectangle (5.5,0.25);
\filldraw [fill=blue, draw=black] (5.5,0.0) rectangle (6.0,0.25);

\filldraw [fill=red, draw=black] (6.0,0.0) rectangle (6.5,0.25);
\filldraw [fill=blue, draw=black] (6.5,0.0) rectangle (7.0,0.25);
\filldraw [fill=blue, draw=black] (7.0,0.0) rectangle (7.5,0.25);
\filldraw [fill=red, draw=black] (7.5,0.0) rectangle (8.0,0.25);

\filldraw [fill=blue, draw=black] (8.0,0.0) rectangle (8.5,0.25);
\filldraw [fill=red, draw=black] (8.5,0.0) rectangle (9.0,0.25);
\filldraw [fill=red, draw=black] (9.0,0.0) rectangle (9.5,0.25);
\filldraw [fill=blue, draw=black] (9.5,0.0) rectangle (10.0,0.25);

\filldraw [fill=red, draw=black] (10.0,0.0) rectangle (10.5,0.25);
\filldraw [fill=blue, draw=black] (10.5,0.0) rectangle (11.0,0.25);
\filldraw [fill=blue, draw=black] (11.0,0.0) rectangle (11.5,0.25);
\filldraw [fill=red, draw=black] (11.5,0.0) rectangle (12.0,0.25);

\filldraw [fill=red, draw=black] (12.0,0.0) rectangle (12.5,0.25);
\filldraw [fill=blue, draw=black] (12.5,0.0) rectangle (13.0,0.25);
\filldraw [fill=blue, draw=black] (13.0,0.0) rectangle (13.5,0.25);
\filldraw [fill=red, draw=black] (13.5,0.0) rectangle (14.0,0.25);

\filldraw [fill=blue, draw=black] (14.0,0.0) rectangle (14.5,0.25);
\filldraw [fill=red, draw=black] (14.5,0.0) rectangle (15.0,0.25);
\filldraw [fill=red, draw=black] (15.0,0.0) rectangle (15.5,0.25);
\filldraw [fill=blue, draw=black] (15.5,0.0) rectangle (16.0,0.25);

\node  (0) at (0.25,0.5) {$0$}; 
\node  (1) [right of=0] {$1$};
\node  (2) [right of=1] {$1$}; 
\node  (3) [right of=2] {$0$}; 

\node  (4) [right of=3] {$1$}; 
\node  (5) [right of=4] {$0$}; 
\node  (6) [right of=5] {$0$}; 
\node  (7) [right of=6] {$1$}; 

\node  (8) [right of=7] {$1$}; 
\node  (9) [right of=8] {$0$}; 
\node  (10) [right of=9] {$0$}; 
\node  (11) [right of=10] {$1$}; 

\node  (12) [right of=11] {$0$}; 
\node  (13) [right of=12] {$1$}; 
\node  (14) [right of=13] {$1$}; 
\node  (15) [right of=14] {$0$};

\node  (16) [right of=15] {$1$};
\node  (17) [right of=16] {$0$};
\node  (18) [right of=17] {$0$};
\node  (19) [right of=18] {$1$};

\node  (20) [right of=19] {$0$};
\node  (21) [right of=20] {$1$};
\node  (22) [right of=21] {$1$};
\node  (23) [right of=22] {$0$};

\node  (24) [right of=23] {$0$};
\node  (25) [right of=24] {$1$};
\node  (26) [right of=25] {$1$};
\node  (27) [right of=26] {$0$};

\node  (28) [right of=27] {$1$};
\node  (29) [right of=28] {$0$};
\node  (30) [right of=29] {$0$};
\node  (31) [right of=30] {$1$};

\end{tikzpicture}

\caption{A portion of the tiling associated with a Thue-Morse sequence.}

\label{FIG:tiling}

\end{figure}
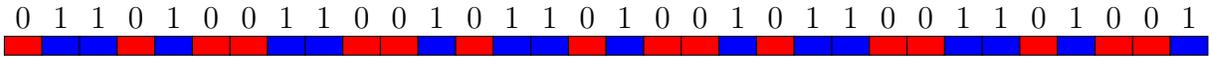
		
		If $\phi$ is primitive and aperiodic, then $\Omega_\phi$ is a compact connected metric space which is nowhere locally connected. The tiling space $\Omega_\phi$ can be loosely seen as a twisted product of a Cantor set with the circle\footnote{In topological language, $\Omega_\phi$ is a fiber bundle over the circle with Cantor set fibers.}. The natural translation $T\mapsto T+t$ for $t\in\mathbb{R}$ equips $\Omega_\phi$ with a continuous $\mathbb{R}$-action which is minimal (in the sense that all $\mathbb{R}$-orbits are dense) as long as $\phi$ is primitive. In this respect, tiling spaces are closely related to the more well-known spaces, the solenoids.
		
		\begin{definition}
		For substitutions $\phi$ and $\eta$, we say that a homeomorphism $f \colon X_\phi \to X_\eta$ is a \emph{topological conjugacy} if $f \circ \sigma = \sigma \circ f$, where $\sigma$ is the shift map on the respective subshifts. If such a topological conjugacy exists, we say that $\phi$ and $\eta$ are \emph{conjugate} substitutions.
		\end{definition}
		\begin{example}
		Any letter-for-letter relabelling of a substitution describes a topological conjugacy. Suppose $\phi$ is the Fibonacci substitution $\phi \colon 0 \mapsto 01, 1 \mapsto 0$ as in Example \ref{EX:fibonacci}. Also consider the substitution $\eta \colon a \mapsto b, b \mapsto ba$. We can describe a rule $0 \mapsto b, 1 \mapsto a$ which extends globally to a topological conjugacy $X_\phi \to X_\eta$.
		\end{example}
		\begin{figure}[H]
\centering
\includegraphics[scale=0.5]{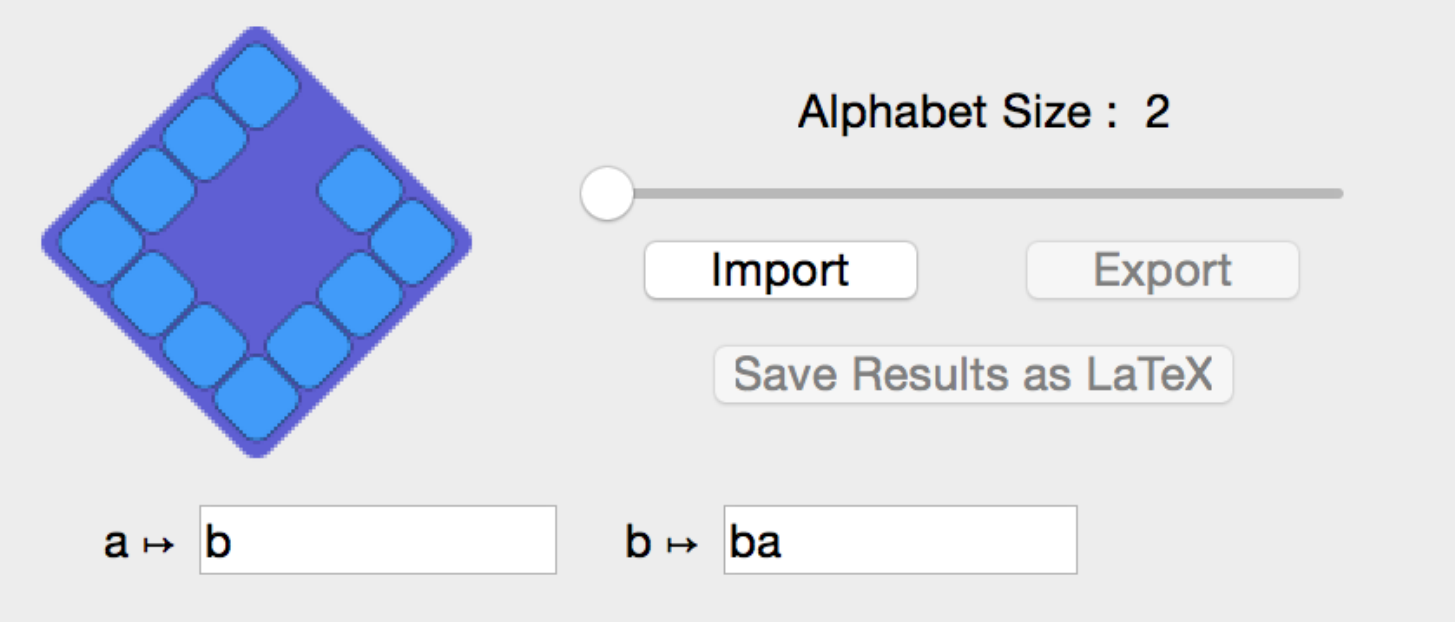}
\caption{The input of the substitution $a \mapsto b, b \mapsto ba$ into Grout.}
		\end{figure}
		Conjugacy is a strong form of equivalence for substitutions and it is an important task to be able to decide when two substitutions are conjugate. The next proposition hints that perhaps an easier object to study is the tiling space associated with a substitution.
		\begin{proposition}
		If $\phi$ and $\eta$ are conjugate, then $\Omega_\phi$ and $\Omega_\eta$ are homeomorphic.
		\end{proposition}
		The proof is an immediate consequence of the definition of the tiling space. The converse is false---as a simple counterexample, consider the periodic substitutions $a \mapsto ab, b \mapsto ab$ and $a \mapsto aab, b \mapsto aab$ whose tiling spaces are both homeomorphic to circles, but whose subshifts do not even have the same cardinality (2 for the former, 3 for the latter).
		
		We will revisit the homeomorphism type of a tiling space and algebraic invariants of such spaces in the section on cohomology, Section \ref{SEC:cohomology}.

\section{Combinatorics and geometry}\label{SEC:main}

\subsection{Substitution matrices and their properties}\label{SEC:matrix}

After primitivity, the next properties that can be calculated from a substitution matrix are the natural tile frequencies and tile lengths of the substitution \cite{S:frequency}.  This requires us to compute the eigenvalues of the matrix and make use of the Perron-Frobenius theorem \cite{LM:introduction-to-symbolic}.

\begin{proposition}[Perron-Frobenius]
	Let $M$ be a primitive matrix.
	\begin{enumerate}
	 \item[(i)] There is a positive real number $\lambda_{PF}$, called the \emph{Perron-Frobenius eigenvalue}, such that $\lambda_{PF}$ is a simple eigenvalue of $M$ and any other eigenvalue $\lambda$ is such that $|\lambda|<\lambda_{PF}$.
	 \item[(ii)] There exist left and right eigenvectors, called the \emph{left and right Perron-Frobenius eigenvectors}, $\mathbf{l}_{PF}$ and $\mathbf{r}_{PF}$ associated with $\lambda_{PF}$ whose entries are all positive and which are unique up to scaling.
	\end{enumerate}

\end{proposition}
Given the above theorem, it is natural to ask what information is contained in the PF (Perron-Frobenius) eigenvalue and eigenvectors of $M_\phi$ for a primitive substitution $\phi$.

By assigning a length to the tiles labelled by each letter, we hope for such a length assignment to behave well with the given substitution. The left PF eigenvector offers a natural choice of length assignments. If we assign the length $(\mathbf{l}_{PF})_i$ to the letter $i$, the $i$th component of the left PF eigenvector, then we can replace our combinatorial substitution by a geometric substitution---this is sometimes called an \emph{inflation rule} \cite{BG:book}. This geometric substitution expands the tile with label $i$ by a factor of $\lambda_{PF}$ and then partitions this new interval into tiles of prescribed lengths and labels according to the symbolic substitution. In order to give a unique output, Grout normalizes the left PF eigenvector so that the smallest entry is $1$.

The information contained in the right PF eigenvector is also useful. The right PF eigenvector, once normalized so that the sum of the entries is $1$, gives the relative frequencies of each of the letters appearing in any particular bi-infinite sequence which is admitted by $\phi$. That is, if $|w|$ is the length of the word $w$, $|w|_i$ is the number of times the letter $i$ appears in the word $w$, and letting $s_{[-k,k]} = s_{-k} \cdots s_{-1} s_0 s_1 \cdots s_k$, then $$\lim_{k \to \infty} |s_{[-k,k]}|_i/|s_{[-k,k]}| = (\mathbf{r}_{PF})_i$$ for any $s \in X_\phi$.
\begin{example}\label{EX:rudin-shapiro}
Let $\Al = \{0,1,2,3\}$ and let $\phi \colon \Al \to \Al^+$ be given by $\phi \colon 0 \mapsto 01, 1 \mapsto 02, 2 \mapsto 31, 3 \mapsto 32$ whose limit sequence with seed 0 is the fixed point of the \emph{Rudin-Shapiro substitution} \seqnum{A100260} \cite{R:rudin-shapiro}, (see also \seqnum{A073057}, \seqnum{A020985}):
$$0102 0131 0102 3202 0102 0131 3231 0102 \cdots,$$
which is related to the \emph{paper-folding sequence} \seqnum{A014577}, (see also \seqnum{A014709}, \seqnum{A014710}, \seqnum{A106665}).
The substitution matrix for $\phi$ has PF eigenvalue $\lambda_{PF} = 2$ and PF eigenvectors $\mathbf{l}_{PF} = (1,1,1,1)$ and $\mathbf{r}_{PF} = (\frac{1}{4},\frac{1}{4},\frac{1}{4},\frac{1}{4})$. So, from the left PF eigenvector, we assign unit length to all four tiles. From the normalized right PF eigenvector, we can tell that each letter appears with equal frequency of $\frac{1}{4}$.
\end{example}
The PF eigenvalues of a substitution matrix will play a role in Section \ref{SEC:pisot} when we introduce Pisot substitutions. Eigenvalues and entries of eigenvectors are often irrational. Grout displays an approximation of these values, as shown in Figure \ref{FIG:matrix}.

\begin{figure}[h!]
\centering
\includegraphics[scale=0.5]{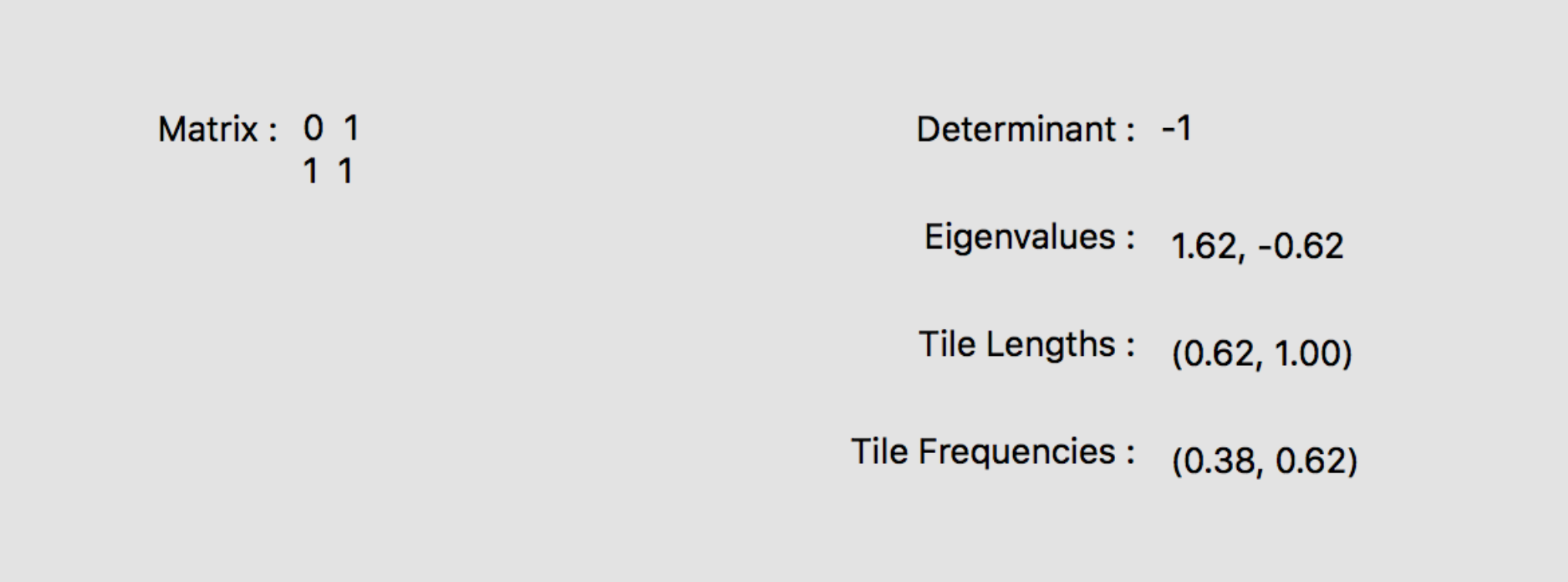}
\caption{The results of matrix calculations in Grout for the substitution $a \mapsto b, b\mapsto ba$.}
\label{FIG:matrix}
\end{figure}

\subsection{Enumerating \texorpdfstring{$n$}{n}-letter words}

Now that we have introduced the basic structure of the substitutions, and discussed the problem of primitivity and other matrix related calculations, we can discuss the first main combinatorial function in Grout.

\begin{definition}
Given a substitution $\phi \colon \mathcal{A} \to \mathcal{A}^+$, we define the \emph{complexity function} at $n$ to be the number of unique $n$-letter words admitted by $\phi$.  We let $p_\phi$ denote this function, and so $p_\phi(n) = |\mathcal{L}^n_\phi|$. If $s$ is a sequence, we similarly let $p_s$ denote the complexity function of $s$, where $p_s(n)$ is the number of unique length-$n$ factors of $s$.

In the literature, one also encounter the terms \emph{subword complexity} or \emph{factor complexity}.
\end{definition}
Clearly, if $\phi$ is primitive, then for all $s \in X_\phi$ we have $p_s(n) = p_\phi(n)$.
\begin{example}
The period doubling sequence from Example \ref{EX:period-doubling} is found to have a complexity function \seqnum{A275202} whose first few terms are $$p_\phi(n) \colon 2, 3, 5, 6, 8, 10, 11, 12, 14, 16, 18, 20, 21, 22, 23, 24, 26, 28, 30, 32, 34, 36, 38, 40, 41, 42, \ldots.$$
\end{example}

Complexity functions are clearly non-decreasing, and for an alphabet $\mathcal{A}$ on $l$ letters, we have $p_s(n) \leq l^n$. The celebrated Morse-Hedlund theorem \cite{MH:complexity} says that a bi-infinite sequence $s$ is periodic if and only if $p_s(n)$ is eventually constant if and only if there exists $n \geq 1$ such that $p_s(n) \leq n$. In particular, for aperiodic sequences one knows that $p_s(n) \geq n+1$ for all natural numbers $n$. The non-periodic sequences which attain this minimal complexity are called the \emph{Sturmian} sequences \cite{F:book}. The complexity function of a sequence is a useful measure of disorder in a sequence. The asymptotic growth rate gives topological and dynamical information about the subshift and tiling space of a sequence \cite{F:complexity} and, in particular, is a topological invariant \cite{J:complexity}.

One is usually interested in either a deterministic formula for $p_\phi$ or information about the asymptotic growth rate of $p_\phi$ such as polynomial degree; Grout can be used to at least give circumstantial evidence for these, though has no means of calculating either (this appears to be a very difficult problem in general). A useful survey article by Ferenczi collates many of the key results on complexity functions \cite{F:complexity2}.

Of particular interest are the number of two-letter and three-letter words admitted by a substitution, as we use them later to compute cohomology.  For this reason, Grout not only enumerates the number of $n$-letter words, but prints out the two-letter and three-letter words as required.

The algorithm to find $n$-letter admitted words begins by generating a seed word\footnote{By primitivity, we can begin with any letter as a seed, so the choice of $0$ here is arbitrary.} $w = \phi^k(0)$ for some $k \geq 1$ such that $|w|\geq n$, and count all unique $n$-letter words appearing as factors of the seed.  We then iterate $\phi$ on $w$ and add all new $n$-letter words appearing in $\phi(w)$ to the result.  At each stage of the iteration, we count the size before and after adding the new words. If our list of admitted words does not change between iterates, we can stop as no new $n$-letter words are generated after a step that generates no new $n$-letter words.  It follows that the value $p_\phi(n)$ is computable in finite time for any fixed $n \geq 1$. Moreover, it is efficient to compute successive terms of the complexity function of a primitive substitution \cite{GSS:subword-complexity}.

\begin{example}
It is well known that the complexity function for the Fibonacci substitution satisfies $p_\phi(n)=n+1$ (so the Fibonacci word is Sturmian), and we can verify this for any value of $n$ by computing its complexity in Grout.

\begin{figure}[H]
\centering
\includegraphics[scale=0.5]{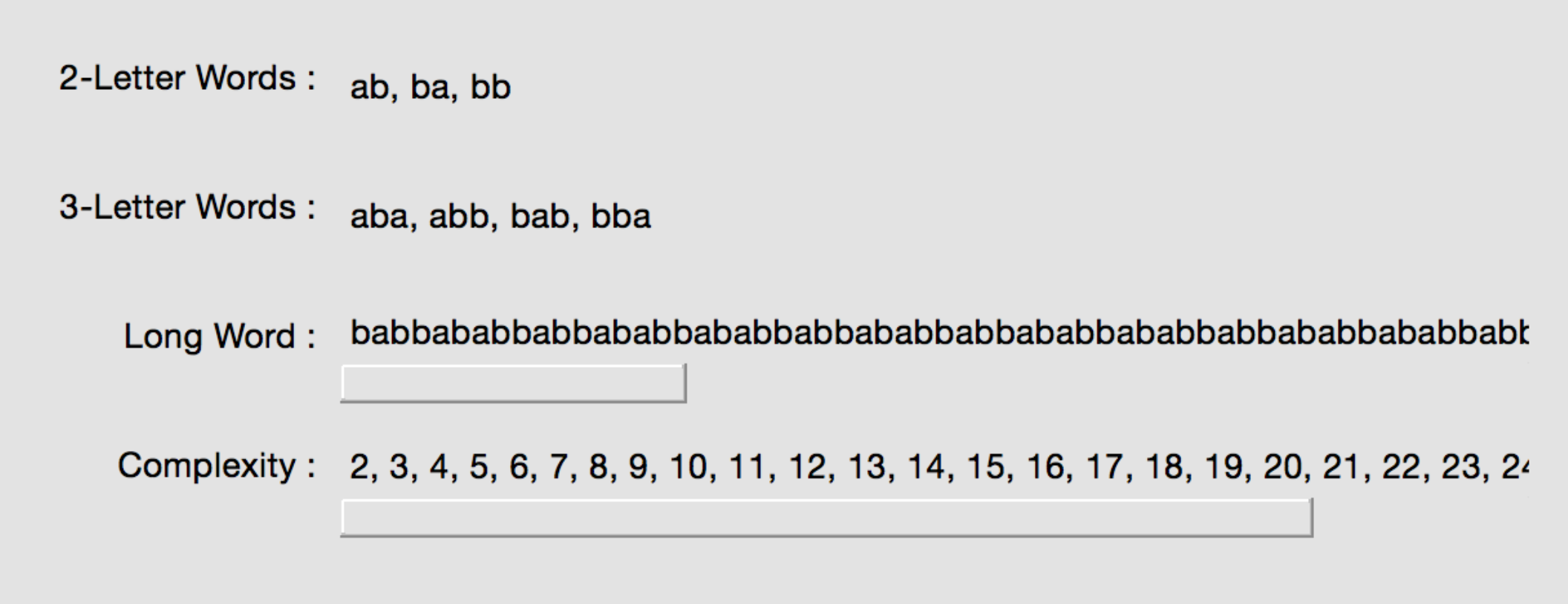}
\caption{The words results for the substitution $a \mapsto b, b\mapsto ba$ displayed in Grout.}
\end{figure}
\end{example}

\subsection{Barge-Diamond and Anderson-Putnam complexes}\label{SEC:complexes}
Grout has the ability to output two graphs as pdf files (provided that the user has PDFLaTeX installed).  The first of these is the \emph{Barge-Diamond complex} \cite{BD}, which is the key tool used in one of the methods of computing the \Cech cohomology of the tiling space.

\begin{definition}
Let $\mathcal{A} = \{0, \ldots , (l-1)\}$ be an alphabet with a substitution $\phi \colon \mathcal{A} \to \mathcal{A}^+$. We construct the \emph{Barge-Diamond complex} of $\phi$ (\emph{BD complex} for short) as follows:  we have two vertices for each $i \in \Al$; an \emph{in} node $v_i^+$ and an \emph{out} node $v_i^-$.  We draw an edge from $v_i^+$ to $v_i^-$ for all $i$.  Then, for all two-letter words $ij \in \mathcal{L}^2_\phi$ admitted by $\phi$, we draw an edge from $v_i^-$ to $v_j^+$.
\end{definition}

\begin{example}\label{EX:platinum}
The \emph{platinum mean} \cite[Ex.\ 4.7, p.\ 93]{BG:book} is the algebraic number $\lambda_{pm} = 2 + \sqrt{3}$, which is the PF eigenvalue of the matrix $\left[\begin{smallmatrix} 3 & 2 \\ 1 & 1 \end{smallmatrix}\right]$ (this matrix is obviously not unique). For this reason, we call the substitution $\phi \colon 0 \mapsto 0001, 1 \mapsto 001$ the \emph{platinum mean substitution} with fixed point \seqnum{A275855}:
$$0 0 0 1 0 0 0 1 0 0 0 1 0 0 1 0 0 0 1 0 0 0 1 0 0 0 1 0 0 1 0 0 0 1 0 0 0 1 0 0 0 1 0 0 1 0 0 0 1 0 0 0 1 0 0 1 0 0 0 1 0 0 0 1 0 0 0 1 0 0 1 \cdots.$$
The two-letter and three-letter admitted words for $\phi$ are $\mathcal{L}^2_\phi = \{00, 01, 10\}$ and $\mathcal{L}^3_\phi = \{000,001,010,100\}$. The BD complex for $\phi$ is given in Figure \ref{FIG:BD}.

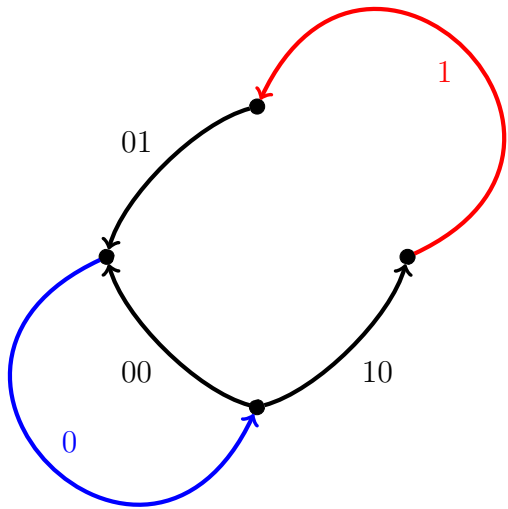
\begin{figure}[H]
\centering

\begin{tikzpicture}[->, node distance=2cm, auto] 
\clip (-5,-3.4) rectangle (5,3.5);
\node [fill,circle,draw,inner sep = 0pt, outer sep = 0pt, minimum size=2mm] (ai) at (2.000000,0.000000) {}; 
\node [fill,circle,draw,inner sep = 0pt, outer sep = 0pt, minimum size=2mm] (ao) at (0.000593,2.000000) {}; 
\node [fill,circle,draw,inner sep = 0pt, outer sep = 0pt, minimum size=2mm] (bi) at (-2.000000,0.001185) {}; 
\node [fill,circle,draw,inner sep = 0pt, outer sep = 0pt, minimum size=2mm] (bo) at (-0.001778,-1.999999) {}; 
\draw (ai) edge[bend right=110, looseness=3, ->, ultra thick, color=red] node {$1$}(ao); 
\draw (bi) edge[bend right=110, looseness=3, ->, ultra thick, color=blue] node {$0$}(bo); 
\draw [-,ultra thick, bend right, draw=white, line width=6pt, looseness=0.7] (bo) to (ai); 
\draw [->,ultra thick, bend right, looseness=0.7, swap] (bo) to node {$10$} (ai); 
\draw [-,ultra thick, bend right, draw=white, line width=6pt, looseness=0.7] (ao) to (bi); 
\draw [->,ultra thick, bend right, looseness=0.7, swap] (ao) to node {$01$} (bi); 
\draw [-,ultra thick, bend right, draw=white, line width=6pt, looseness=0.7] (bo) to (bi); 
\draw [->,ultra thick, bend left, looseness=0.7] (bo) to node {$00$} (bi); 

\end{tikzpicture}
\caption{The Barge-Diamond complex for the platinum mean substitution.}
\label{FIG:BD}
\end{figure}
\end{example}

The other complex that we consider for a substitution is (a variant of) the \emph{collared Anderson-Putnam complex} \cite{AP}, which is again the key tool used in one of the methods of computing \Cech cohomology of the tiling space. This particular definition is based on what G\"{a}hler and Maloney call the \emph{Modified Anderson-Putnam complex} \cite{GM:multi-one-d}.  The Anderson-Putnam complex is constructed in a similar fashion to the Barge-Diamond complex, but makes use of both the two-letter and three-letter words.

\begin{definition}\label{DEF:ap-complex}
Let $\mathcal{A} = \{0, \ldots , (l-1)\}$ be an alphabet with a substitution $\phi \colon \mathcal{A} \to \mathcal{A}^+$. We construct the \emph{Anderson-Putnam complex} of $\phi$ (\emph{AP complex} for short) as follows:  we have a vertex $v_{ij}$ for each two-letter word $ij \in \mathcal{L}^2_\phi$ admitted by $\phi$.  We draw an edge from $v_{ij}$ to $v_{jk}$ if and only if the three-letter word $ijk$ is admitted by $\phi$.
\end{definition}

\begin{remark}
Readers familiar with the family of Rauzy graphs \cite{R:rauzy-graphs} associated with an infinite sequence over a finite alphabet will recognize a similarity between the Anderson-Putnam complex and the \emph{second Rauzy graph} of an infinite sequence. Indeed, they are isomorphic as graphs. Anderson and Putnam originally introduced their complex in the context of tilings of $d$-dimensional Euclidean space, where the case $d=1$ corresponds to tilings of the line. Such one-dimensional tilings are combinatorially described by bi-infinite sequences over finite alphabets. In their definition, the edges of the AP complex also come equipped with data on the lengths of the corresponding tiles in the associated tiling. Although useful in other contexts, the cohomology groups that we later study via the AP complex are independent of this edge-length data.
\end{remark}

\begin{remark}
	One should note that this definition of the AP complex is slightly different from the definition originally introduced by Anderson and Putnam. In particular, the original definition distinguishes between different occurrences of a two-letter word $ij$ if the occurrences of three-letter words containing $ij$ as a factor do not overlap on some admitted four-letter word. For example, if the language of a substitution includes the two-letter word $ab$, the three-letter words $xab, yab, abw, abz$, and the four-letter words $xabw, yabz$, but the words $xabz$ and $yabw$ do not belong to $\mathcal{L}_\phi$, then the original definition of the AP complex has two instances of vertices with the label $ab$, say $(ab)_1$ and $(ab)_2$. In our definition, these vertices are identified, so that $(ab)_1 \sim (ab)_2 \sim ab$. An example of such a substitution is given by $\phi \colon a \mapsto bc, b \mapsto baab, c \mapsto caac$ where we label exactly one vertex with the label $aa$, but the original definition would require we include two distinct vertices labelled $(aa)_1$ and $(aa)_2$.
\end{remark}
In our discussion of cohomology calculated via AP complexes in Section \ref{SEC:ap-cohomology}, we use the version of the AP complex given in Definition  \ref{DEF:ap-complex} to describe the performed calculations. This is also the method implemented for Grout. It would have been possible to use the original definition, or one of the many variant AP complexes that have been defined in the literature. There are at least three such variants discussed by G\"{a}hler and Maloney \cite{GM:multi-one-d}, of varying complexities and situations in which they can be used.

\begin{example}
As discussed in Example \ref{EX:platinum}, the platinum mean substitution has two-letter and three-letter admitted words $\mathcal{L}^2_\phi = \{00, 01, 10\}$ and $\mathcal{L}^3_\phi = \{000, 001, 010, 100\}$. The AP complex for the platinum mean substitution is given in Figure \ref{FIG:AP}.

\begin{figure}[H]

\centering

\begin{tikzpicture}[->, node distance=3cm, auto, text=black, line width=0.5mm] 
\clip (-3,-4) rectangle (2.4,2);
\node [circle,inner sep = 0pt, outer sep = 2pt, minimum size=2mm] (abi) at (2.000000,0.000000) {$10$}; 
\node [circle,inner sep = 0pt, outer sep = 2pt, minimum size=2mm] (bai) at (-0.999316,1.732446) {$01$}; 
\node [circle,inner sep = 0pt, outer sep = 2pt, minimum size=2mm] (bbi) at (-1.001368,-1.731260) {$00$}; 
\draw [-, bend left, draw=white, line width=2pt, looseness=0.7] (abi) to (bbi); 
\draw [->, bend left, looseness=0.7] (abi) to (bbi); 
\draw [-, bend left, draw=white, line width=2pt, looseness=0.7] (bai) to (abi); 
\draw [->, bend left, looseness=0.7] (bai) to (abi); 
\draw [-, bend left, draw=white, line width=2pt, looseness=0.7] (bbi) to (bai); 
\draw [->, bend left, looseness=0.7] (bbi) to (bai);
\draw (bbi) to [out=300,in=180,looseness=20] (bbi);

\end{tikzpicture}

\caption{The Anderson-Putnam complex for the platinum mean substitution.}
\label{FIG:AP}
\end{figure}
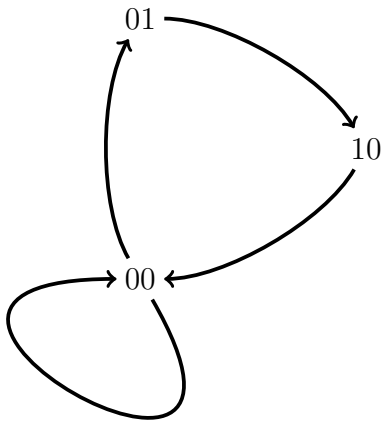
\end{example}

\subsection{Recognizability}\label{SEC:recog}

	The results in this section on recognizability are not new, but are rarely cited in the literature. It follows that this section may be of interest to experts who may be unaware of such results, especially given the strength of Corollary \ref{COR:recog}.
\begin{definition}	
		Let $\phi$ be a substitution on the alphabet $\mathcal{A}$. We say $\phi$ is \emph{recognizable} if, for every bi-infinite sequence $s \in X_\phi$ admitted by $\phi$, there is a unique way of decomposing $s$ into words of the form $\phi(a)$ for $a \in\mathcal{A}$. That is, there exists a unique bi-infinite sequence $s' = \cdots a_{-1} a_0 a_1 \cdots$ and integer $n<|\phi(a_0)|$ such that $s = \sigma^n(\phi(s'))$. In this sense, one can \emph{recognize} which substituted letter each letter in $s$ has come from.
		
		Equivalently, we say $\phi$ is recognizable if there exists a natural number $K \geq 1$ such that for all admitted words $w \in \mathcal{L}_\phi$ with $|w| > 2K$, there exist unique words $x, y$ of length $|x|, |y| \leq K$ and a unique admitted word $u \in \mathcal{L}_\phi$ such that $w = x \phi(u) y$.
	\end{definition}
Recognizability is an important property of a tiling, as many of the tools used to study the topology of the associated tiling spaces rely on recognizability as a hypothesis, much like primitivity. Recognizability of a primitive substitution is equivalent to aperiodicity of the subshift $X_\phi$ by a celebrated result of Moss\'{e} \cite{M:aperiodic} (later extended to higher dimensions by Solomyak \cite{S:aperiodic}), and so any recognizability check for substitutions will also constitute an aperiodicity check for fixed points of primitive substitutions.  The algorithm designed to determine if a given substitution is recognizable relies on finding a fixed letter and return words to that fixed letter.

\begin{definition}
Given a substitution $\phi$ on an alphabet $\mathcal{A}$, the letter $a$ is said to be \emph{fixed} (on the left) of \emph{order} $k$ if there exists some integer $k$ such that $\phi^k(a) = a u$ for some word $u$. Every substitution has at least one fixed letter and the value of $k$ for such a letter is bounded by the size of the alphabet.
\end{definition}

\begin{definition}
Given a fixed letter $a$, a \emph{return word} to $a$ is a word $v$ such that $v = au $ for some (possibly empty) word $u \in (\mathcal{A}\setminus\{a\})^\ast$, and $va = aua$ is an admitted word of the substitution.
\end{definition}

Since their introduction by Durand \cite{D:derived-seqs}, return words have proved to be an extremely useful tool in studying substitutions. If $\phi$ is primitive, then the set of return words to any letter is finite. We use these return words to determine whether a substitution is recognizable or not.

The following is a result of Harju and Linna \cite{HL:periodicity}.
	\begin{theorem}
		Let $\phi$ be a primitive substitution on $\mathcal{A}$ and let $a$ be a fixed letter. Let $\mathcal{R}_{\phi,a}$ be the set of all return words to $a$. So $\mathcal{R}_{\phi,a} = \{v \mid v = au,\: aua \in \mathcal{L}_\phi,\: u \in (\mathcal{A} \setminus \{a\})^\ast\}$. The substitution $\phi$ is not recognizable if and only if, for all $v, v' \in \mathcal{R}_{\phi,a}$, there exists a natural number $p \geq 1$ such that $\phi^p(vv') = \phi^p(v'v)$.
	\end{theorem}

	The set of return words $\mathcal{R}_{\phi,a}$ is finite, but a priori one still needs to check arbitrarily high values of $p$ according to the above theorem. This, in fact, is not the case. The next result appears in the work of Culik \cite{C:iterates-of-words} and also Ehrenfeucht and Rozenberg \cite{ER:iterates-of-words}.
	\begin{lemma}
		Let $\phi$ be a substitution on $\mathcal{A}$ and let $|\mathcal{A}| = l$. For words $u, w \in \mathcal{A}^+$, there exists a natural number $p \geq 1$ such that $\phi^p(u) = \phi^p(w)$ if and only if $\phi^l(u) = \phi^l(w)$.
	\end{lemma}
	That is, if some iterated substitution of $u$ and $w$ are ever equal, then their iterates must become equal at least by the $l$th iteration of the substitution, where $l$ is the size of the alphabet.
	Together, these results give us the following corollary which provides us with a finite deterministic check for recognizability:
	\begin{corollary}\label{COR:recog}
	Let $\phi$ be a primitive substitution on an alphabet with $l$ letters. Let $a \in \mathcal{A}$ be a fixed letter. The substitution $\phi$ is recognizable if and only if, for all return words $v, v'$ to $a$ $$\phi^l(vv') = \phi^l(v'v) \implies v = v'.$$
	\end{corollary}

\begin{figure}[H]
\centering
\includegraphics[scale=0.5]{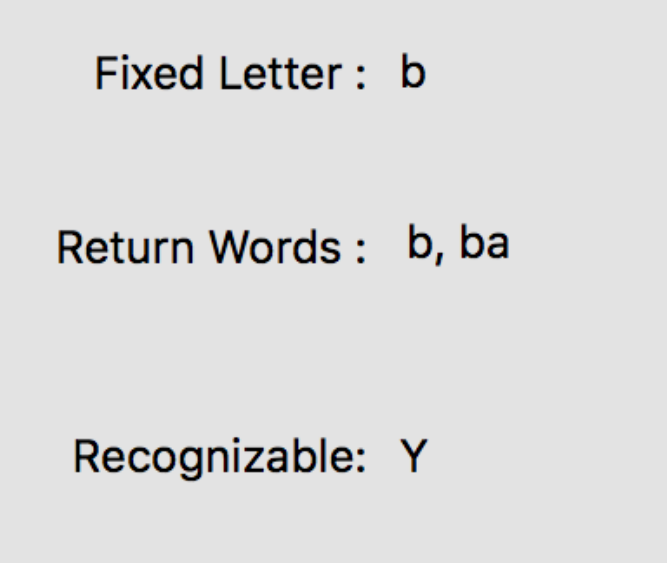}
\caption{The recognizability results in Grout for the substitution $a \mapsto b, b \mapsto ba$.}
\end{figure}
\begin{example}\label{EX:TM-return}
Let $\phi \colon 0 \mapsto 01, 1 \mapsto 10$ be the Thue-Morse substitution. The letter $0$ is a fixed letter and the return words to $0$ are $\mathcal{R}_{\phi,0} = \{0, 01, 011\}$. We check all three distinct pairs:

$$\begin{array}{ll}
\phi^2(0\:01) = 011001101001 				& \phi^2(01\:0) = 011010010110\\
\phi^2(0\:011) = 0110 011010011001 			& \phi^2(011\:0) = 011010011001 0110\\
\phi^2(01\:011) = 01101001 011010011001 	& \phi^2(011\:01) =  011010011001 01101001.
\end{array}$$
As no row produces an equal pair of words, we conclude from Corollary \ref{COR:recog} that the Thue-Morse substitution is recognizable.
\end{example}

\section{Cohomology of tiling spaces}\label{SEC:cohomology}
In this section, we discuss methods for calculating the first \Cech cohomology of the tiling space associated with a primitive recognizable substitution. In order to understand the topological structure of these tiling spaces, one would ordinarily need to be familiar with the notion of an inverse limit of topological spaces. Indeed, to prove that the following methods are correct, inverse limits are a necessary intermediary. Our approach, however, omits any discussion of inverse limits, and we instead refer the interested reader to the very approachable book by Sadun \cite{S:book}. We replace such discussion instead with combinatorial and algebraic machinery.

We first introduce two important tools from the theory of abelian groups and homological algebra---\emph{exact sequences} and \emph{direct limits}. These notions and results can be found in any introductory text on homological algebra \cite{CE:book}.
\begin{definition}
Let $A_i$ be a bi-infinite sequence of abelian groups for $i \in \mathbb{Z}$ and let $F = \{f_i \colon A_i \to A_{i+1}\}$ be a sequence of group homomorphisms. As a diagram, we represent such a sequence of homomorphisms as follows:
$$\cdots \to A_{i} \stackrel{f_{i}}{\to} A_{i+1} \stackrel{f_{i+1}}{\to} A_{i+2} \to \cdots.$$
We say that the sequence of homomorphisms $F$ is \emph{exact}, and call the above diagram a \emph{long exact sequence} (LES) if, for all $i \in \mathbb{Z}$, $\ker f_{i+1} = \operatorname{im} f_i$.

If all but a consecutive collection of three of the groups $A_i, A_{i+1}, A_{i+2}$ are trivial, we say that the diagram is a \emph{short exact sequence} (SES). Generally, an SES will be written as a finite diagram capped by trivial groups:
$$0 \to A \to B \to C \to 0.$$
\end{definition}
The following result is an easy exercise.
\begin{proposition}
The diagram $0 \stackrel{f}{\to} A \stackrel{g}{\to} B \stackrel{h}{\to} 0$ is exact if and only if $g$ is an isomorphism.
\end{proposition}

The next result appears in more general forms, but we give here only the version that we require. The proof is well-known \cite{CE:book}.

Let $0 \to A \stackrel{f}{\to} B \stackrel{g}{\to} C \to 0$ be a short exact sequence of abelian groups. If there exists a homomorphism $i \colon C \to B$ such that $g \circ i = \operatorname{Id}_C$, or there exists a homomorphism $j\colon A \to B$ such that $j \circ f = \operatorname{Id}_A$, then we say that the SES \emph{splits}.

\begin{lemma}[Splitting lemma]
If $0 \to A \to B \to C \to 0$ is a split SES, then there exists an isomorphism $h \colon B \to A \oplus C$.
\end{lemma}
It is useful to establish criteria under which a short exact sequence splits. One such criterion is that the third non-trivial group in the sequence is free abelian.
\begin{proposition}
If $0 \to A \to B \stackrel{f}{\to} \mathbb{Z}^n \to 0$ is an SES, then it is split. Further, $B \cong A \oplus \mathbb{Z}^n$.
\end{proposition}
\begin{proof}
As the sequence is exact, the image of $f$ is the kernel of a homomorphism whose image is trivial, hence $\operatorname{im}f = \mathbb{Z}^n$. Let $\{e_j \mid 1,
\ldots, n \}$ be a set of generating elements of $\mathbb{Z}^n$ and let $b_1, \ldots, b_n$ be such that $f(b_j) = e_j$. Define the map $i \colon \mathbb{Z}^n \to B$ on generators by $i(e_j) = b_j$ and extend this to a homomorphism. We have $f (i (e_j)) = f(b_j) = e_j$ for all $j$ and so $f \circ i = \operatorname{Id}_{\mathbb{Z}^n}$. Hence, the SES splits and, by the splitting lemma, we conclude that $B \cong A \oplus \mathbb{Z}^n$.
\end{proof}
The groups which appear as cohomology groups of tiling spaces are usually presented as a so-called \emph{direct limit}.
\begin{definition}
Let $ A_0 \stackrel{f_0}{\to} A_1 \stackrel{f_1}{\to} A_2 \to \cdots$ be an infinite sequence of abelian groups $A_i$ with homomorphisms $f_i$ between them. We let the pair $(A_i,f_i)$ denote this \emph{direct system}. For $j \geq i$, write $f_{ij} = f_j \circ \cdots \circ f_i \colon A_i \to A_{j+1}$ for the composition of a consecutive subsequence of these homomorphisms.

We define the \emph{direct limit} of a direct system $(A_i,f_i)$ to be the disjoint union of the $A_i$ modulo the equivalence relation $a_i \sim a_j$ for $a_i \in A_i$, $a_j \in A_j$ if there exists $N \geq \max\{i,j\}$ such that $f_{iN}(a_i) = f_{jN}(a_j)$. We write
$$\varinjlim(A_i,f_i) = \bigsqcup A_i \bigg/{\sim}.$$
There is an abelian group operation on $\varinjlim(A_i,f_i)$ given by $[a_i]+[a_j] = [f_{iN}(a_i) + f_{jN}(a_j)]$ where $N = \max\{i,j\}$ and square brackets $[-]$ denote an equivalence class with respect to $\sim$.
\end{definition}
The equivalence relation $\sim$ should be thought of as identifying elements in the disjoint union if they are `eventually equal' under applications of the homomorphisms $f_i$ as they move down the direct system of groups. When the context is clear, we may write a direct limit in terms of the homomorphisms appearing in the direct system $\varinjlim (A_i,f_i) = \varinjlim f_i$.

A diagram of commuting homomorphisms $\phi_i \colon A_i \to B_i$ between direct systems $(A_i,f_i)$ and $(B_i,g_i)$ induces a homomorphism $\phi \colon \varinjlim f_i \to \varinjlim g_i$ by the rule $\phi([a_i]) = [\phi_i(a_i)]$.
\begin{example}
If $f_i$ is an isomorphism for all $i$, then $\varinjlim (A_i,f_i) \cong A_i$ for any of the $A_i$. In particular, an isomorphism $h \colon \varinjlim f_i \to A_0$ is given by $h([a_k]) = f_{0k}^{-1}(a_k)$ for $a_k \in A_k$.
\end{example}
\begin{example}
If the homomorphism $f_n$ is the zero homomorphism, then any element in the direct limit that appears before the $n$th position in the sequence is identified with zero by the equivalence relation $\sim$. It follows that if $f_i$ is the zero homomorphism for infinitely many $i$, then $\varinjlim f_i = 0$.
\end{example}
\begin{example}
Suppose $A_i = \mathbb{Z}$ for all $i$ and the homomorphisms are all given by the doubling map $\times 2 \colon n \mapsto 2n$. The direct limit of this direct system is isomorphic to the set of \emph{dyadic rationals} $\mathbb{Z}[1/2]$---the rational numbers whose denominator is a power of $2$ when written in reduced form. The isomorphism $h \colon \mathbb{Z}[1/2] \to \varinjlim (\mathbb{Z},\times 2)$ is given by $h(n/2^k) = [n_k]$ where $n_k$ is the copy of the integer $n$ that appears in the $k$th group $A_k$ in the direct system.
\end{example}
In general, it is difficult to reduce the direct limit of a sequence of matrices acting on free abelian groups to a form which is more familiar; for instance as a direct sum of free abelian parts and $n$-adic rational parts. In fact, it is often the case that this cannot be done \cite{D:decomposable}. However, in certain cases we can give an explicit description in this form.

To an $(n \times n)$-matrix $M$, we can associate a homomorphism $f \colon \mathbb{Z}^n \to \mathbb{Z}^n$ given by the action of the matrix $M$ on the left of column vectors. That is, $f(v) = Mv$. Without loss of generality, we may then write $\varinjlim M = \varinjlim f = \varinjlim (\mathbb{Z}^n,f_i)$ where $f_i = f$ for all $i$ in the case that every map in the direct system is given by the same action of the matrix $M$.
\begin{example}\label{EX:TM-direct-limit}
Suppose $A_i = \mathbb{Z}^2$ for all $i$ and the homomorphisms $f_i$ are given by the action of the matrix $M = \left[\begin{smallmatrix}1 & 1 \\ 1 & 1\end{smallmatrix}\right]$ on the left of column vectors $\left[\begin{smallmatrix}n \\ m\end{smallmatrix}\right] \in \mathbb{Z}^2$. The direct limit $\varinjlim f_i = \varinjlim \left[\begin{smallmatrix}1 & 1 \\ 1 & 1\end{smallmatrix}\right]$ is isomorphic to the dyadic integers $\mathbb{Z}[1/2]$.

We note that each element $\left[\begin{smallmatrix}n \\ m\end{smallmatrix}\right] \in \mathbb{Z}^2$ in the group $A_k$ has a representative in $A_{k+1}$ which lies on the diagonal subgroup $D = \{\left[\begin{smallmatrix}n \\ n\end{smallmatrix}\right] \mid n \in \mathbb{Z}\} < \mathbb{Z}^2$. In particular, $\left[\begin{smallmatrix}1 & 1 \\ 1 & 1\end{smallmatrix}\right] \left[\begin{smallmatrix}n \\ m\end{smallmatrix}\right] = \left[\begin{smallmatrix}n+m \\ n+m \end{smallmatrix}\right]$, and so the direct limit is isomorphic to the direct limit of the induced action on this subgroup $\varinjlim (D,f_i)$ where $f_i\left[\begin{smallmatrix}n \\ n\end{smallmatrix}\right] = \left[\begin{smallmatrix}2n\\ 2n\end{smallmatrix}\right]$. This direct limit is isomorphic to the direct limit of the doubling map on the integers. The isomorphism is given by mapping $\left[\begin{smallmatrix}n \\ n\end{smallmatrix}\right]$ in the $k$th copy of $D$ to $n$ in the $k$th copy of $\mathbb{Z}$. This is illustrated in the following commutative diagram
$$
		\begin{tikzpicture}[node distance=2cm, auto]
  		\node (00) {$\mathbb{Z}$};
  		\node (20) [above of=00] {$D$};
  		\node (10) [above of=20] {$\mathbb{Z}^2$};
  		\node (01) [right of=00] {$\mathbb{Z}$};
  		\node (11) [right of=10] {$\mathbb{Z}^2$};
  		\node (21) [right of=20] {$D$};
  		\node (02) [right of=01] {$\mathbb{Z}$};
  		\node (12) [right of=11] {$\mathbb{Z}^2$};
  		\node (22) [right of=21] {$D$};
  		\node (03) [right of=02] {$\cdots$};
  		\node (13) [right of=12] {$\cdots$};
  		\node (23) [right of=22] {$\cdots$};
  		
  		\node (L2) [left of=20] {$\varinjlim\left(D,\left[\begin{smallmatrix}1 & 1 \\ 1 & 1\end{smallmatrix}\right]\right):$};
  		\node (L1) [left of=10] {$\varinjlim\left(\mathbb{Z}^2,\left[\begin{smallmatrix}1 & 1 \\ 1 & 1\end{smallmatrix}\right]\right):$};
  		\node (L0) [left of=00] {$\varinjlim\left(\mathbb{Z},\times 2\right):$};

  		\draw[->] (10) to node [swap] {} (20);
  		\draw[->] (11) to node [swap] {} (21);
  		\draw[->] (12) to node [swap] {} (22);
  		\draw[->] (20) to node [swap] {} (00);
  		\draw[->] (21) to node [swap] {} (01);
  		\draw[->] (22) to node [swap] {} (02);
  		\draw[->] (L2) to node [swap] {} (L0);
  		\draw[->] (L1) to node [swap] {} (L2);
  		\draw[->] (00) to node {$\times 2$} (01);
  		\draw[->] (01) to node {$\times 2$} (02);
  		\draw[->] (02) to node {$\times 2$} (03);
  		\draw[->] (10) to node {$\left[\begin{smallmatrix}1 & 1 \\ 1 & 1\end{smallmatrix}\right]$} (11);
  		\draw[->] (11) to node {$\left[\begin{smallmatrix}1 & 1 \\ 1 & 1\end{smallmatrix}\right]$} (12);
  		\draw[->] (12) to node {$\left[\begin{smallmatrix}1 & 1 \\ 1 & 1\end{smallmatrix}\right]$} (13);
  		\draw[->] (20) to node {$\left[\begin{smallmatrix}1 & 1 \\ 1 & 1\end{smallmatrix}\right]$} (21);
  		\draw[->] (21) to node {$\left[\begin{smallmatrix}1 & 1 \\ 1 & 1\end{smallmatrix}\right]$} (22);
  		\draw[->] (22) to node {$\left[\begin{smallmatrix}1 & 1 \\ 1 & 1\end{smallmatrix}\right]$} (23); 
		\end{tikzpicture}
$$
where the top row of vertical homomorphisms are given by $\left[\begin{smallmatrix}n \\ m\end{smallmatrix}\right] \mapsto \left[\begin{smallmatrix}n+m \\ n+m\end{smallmatrix}\right]$, and the bottom row of vertical homomorphisms are given by $\left[\begin{smallmatrix}n \\ n\end{smallmatrix}\right] \mapsto n$. The induced homomorphisms on the direct limits are both isomorphisms by the reasoning above.
Hence, from the previous example, we know that $\varinjlim(\mathbb{Z}^2, \left[\begin{smallmatrix}1 & 1 \\ 1 & 1\end{smallmatrix}\right]) \cong \mathbb{Z}[1/2]$; the dyadic rationals.
\end{example}
Even if the direct limit is not easily parsable, we can at least easily extract the rank of a direct limit of a sequence of matrices when the sequence is constant. Here, the rank $\operatorname{rk}A$ of an abelian group $A$ is the dimension of the $\mathbb{Q}$-vector space $A\otimes \mathbb{Q}$, or equivalently, the cardinality of a maximal linearly independent subset.

If $M$ is an $(n\times n)$-matrix, let $p$ be a power such that the rank of the matrix $M^p$ is the same as the rank of $M^{p+1}$ (after which point the ranks of successive powers remain constant). We call the rank $\operatorname{rk} M^p$ the \emph{eventual rank} of $M$ and let $\operatorname{e-rk}M$ denote this value. The next result is an easy exercise in linear algebra and follows from the definition of the direct limit of a matrix.
\begin{proposition}
$$\operatorname{rk}\varinjlim M = \operatorname{e-rk}M.$$
\end{proposition}
There is an entire industry devoted to studying direct limits of integer-valued matrices, their invariants, and classifying such groups, possibly with an additional order structure. We refer the reader to the work of Boyle and Handelman \cite{BH:shift-equivalence} and references therein.

Associated with any topological space $X$ is a collection of groups $\check{H}^k(X)$ for all $k \geq 0$ called the \emph{\Cech cohomology groups} of $X$. The \Cech cohomology of a space is a topological invariant. That is, if $X \cong Y$, then $\check{H}^k(X) \cong \check{H}^k(Y)$ for all $k\geq 0$. In particular, if $\phi$ and $\eta$ are conjugate substitutions, then $\check{H}^k(\Omega_\phi) \cong \check{H}^k(\Omega_\eta)$ for all $k\geq 0$. As we are only studying $1$-dimensional tiling spaces of primitive substitutions, it will always be the case for us that $\check{H}^0(\Omega_\phi) = \mathbb{Z}$ and $\check{H}^k(\Omega_\phi) = 0$ for all $k \geq 2$. Hence, the only interesting group to study is $\check{H}^1(\Omega_\phi)$.

		There are many standard texts providing an introduction to \Cech cohomology \cite{B:bott-tu}. \Cech cohomology is an important topological invariant of tiling spaces and it is of general interest to be able to calculate and study these groups even beyond the task of classifying substitutions up to conjugacy \cite{BGG:substitutions,BBJS:homological-pisot,CS:shape}. Grout implements three different methods for calculating the cohomology of tiling spaces associated with primitive aperiodic substitutions on finite alphabets:
		\begin{enumerate}
			\item The method of Barge-Diamond complexes \cite{BD}.
			\item The method of Anderson-Putnam complexes  \cite{AP}.
			\item The method of forming a conjugate \emph{pre-left-proper} substitution \cite{DHS:return-words}.
		\end{enumerate}
		All three outputs are algebraically equivalent---that is, they represent isomorphic groups---but it is not always obvious that this is the case given the presentations that each method produces. Compare Examples \ref{EX:TM-BD}, \ref{EX:TM-AP}, \ref{EX:TM-proper}, where the Thue-Morse substitution is used as a reference example. These presentations take the form of direct limits of matrices (often of different sizes) and, in the case of the Barge-Diamond method, also direct sums of such groups and quotients by free abelian subgroups. This disparity between presentations of results for the equivalent methods was one of the major motivating factors for developing Grout.  These groups are extremely laborious to calculate by hand for large alphabets unless special criteria are met.
\subsection{Via Barge-Diamond}
	 Let $\phi$ be a primitive, recognizable substitution on the alphabet $\mathcal{A}$. Let $\Gamma_\phi$ be the Barge-Diamond complex of $\phi$ and let $S_\phi$ be the subcomplex of $\Gamma_\phi$ formed by all those edges labelled with two-letter admitted words $ij$.
	
	Let $\tilde{\phi}\colon S_\phi \to S_\phi$ be a graph morphism defined in the following way on vertices: let $l(i)$ and $r(i)$ be the leftmost and rightmost letters of $\phi(i)$ so $\phi(i) = l(i) u r(i)$ (note that $u$ may be empty and $l(i)$ and $r(i)$ could overlap if $|\phi(i)| = 1$). Define $\tilde{\phi}(v_i^+) = v_{l(i)}^+$ and $\tilde{\phi}(v_i^-) = v_{r(i)}^-$. Note that if $ij$ is admitted by $\phi$, then $r(i) l(j)$ is also admitted by $\phi$, and so $\tilde{\phi}$ is a well-defined graph morphism on $S_\phi$. As $\tilde{\phi}(S_\phi) \subset S_\phi$, we can define the eventual range $ER_\phi = \bigcap_{m \geq 0} \tilde{\phi}^m(S_\phi)$ (which stabilizes after finitely many substitutions).
	
	Let $\tilde{H}^0(X) = \mathbb{Z}^{k-1}$ where $k$ is the number of connected components of the space $X$---this group is called the \emph{reduced cohomology} of $X$ in degree zero and always satisfies $\check{H}^0(X) = \tilde{H}^0(X) \oplus \mathbb{Z}$. For this method of computation we make use of the following result attributed to Barge and Diamond \cite{BD}.
	\begin{proposition}
		There is an exact sequence
		$$0 \to \tilde{H}^0(ER_\phi) \stackrel{h_1}{\to} \varinjlim M_\phi^T \stackrel{h_2}{\to} \check{H}^1(\Omega_\phi) \stackrel{h_3}{\to} \check{H}^1(ER_\phi) \to 0.$$
	\end{proposition}
	This exact sequence can be split into the two short exact sequences
	$$0 \to \tilde{H}^0(ER_\phi) \stackrel{h_1}{\to} \varinjlim M_\phi^T \to \operatorname{im} h_2 \to 0$$
	and
	$$0 \to \operatorname{im} h_2 \to \check{H}^1(\Omega_\phi) \stackrel{h_3}{\to} \check{H}^1(ER_\phi) \to 0.$$
	Using the first isomorphism theorem, we know that $\operatorname{im}h_2$ is isomorphic to $\varinjlim M_\phi^T / \ker h_2$. By exactness, $\ker h_2 = \operatorname{im} h_1$ and as $h_1$ is injective by exactness, the image of $h_1$ is isomorphic to its domain. Hence, $\ker h_2 = \operatorname{im} h_1 \cong \tilde{H}^0(ER_\phi)$. It follows that $\operatorname{im}h_2 \cong \varinjlim M_\phi^T / \tilde{H}^0(ER_\phi)$.
	
	The eventual range $ER_\phi$ is a non-empty (possibly disconnected) graph, and so $\tilde{H}^0(ER_\phi)$ and $\check{H}^1(ER_\phi)$ are finitely generated free abelian groups. Using the splitting lemma on the second short exact sequence allows us to write $\check{H}^1(\Omega_\phi)$ as a direct sum.
	\begin{corollary}\label{COR:BD} Let $\phi$ be a substitution on an alphabet $\mathcal{A}$ with $l$ letters. If the eventual range $ER_\phi$ has $k$ connected components and $\check{H}^1(ER_\phi) \cong \mathbb{Z}^m$, then 
		$$\check{H}^1(\Omega_\phi)\cong \varinjlim (\mathbb{Z}^{l},M_\phi^T)/ \mathbb{Z}^{k-1} \oplus \mathbb{Z}^m.$$
	\end{corollary}
	
	Grout displays the \Cech cohomology in the above form when using the Barge-Diamond method.
	
	Recall that the Euler characteristic $\chi$ of a graph $G$ is given by $\chi = E - V$ where $E$ and $V$ are the number of edges and vertices of $G$ respectively. The Euler characteristic is also given as an alternating sum of ranks of cohomology groups; $\chi = \operatorname{rk}\check{H}^1(G) - \operatorname{rk}\check{H}^0(G) = \operatorname{rk}\check{H}^1(G) - \operatorname{rk}\tilde{H}^0(G) - 1$. If $G$ has $k$ connected components, then $\operatorname{rk}\tilde{H}^0(G) = k-1$ and so $m = \operatorname{rk}\check{H}^1(G) = E - V + k$.
	
	When using this method via Corollary \ref{COR:BD}, the only involved part of the calculation is finding the eventual range of the Barge-Diamond complex. The number of connected components of the eventual range then allows us to calculate its cohomology in degree zero. We can then calculate the cohomology in degree one by evaluating the above formula for $m$ derived from the Euler characteristic.
	
	\begin{example}\label{EX:TM-BD}
	Let $\phi \colon 0 \mapsto 01, 1 \mapsto 10$ be the Thue-Morse substitution. The Thue-Morse substitution is primitive and recognizable and has as its two-letter admitted words the set $\mathcal{L}^2_\phi = \{00,01,10,11\}$. The BD complex $\Gamma_\phi$ for $\phi$ is given in Figure \ref{FIG:BD-TM}. We calculate $\tilde{\phi}$ on the four vertices of $\Gamma_\phi$ to be 
	$$\tilde{\phi} \colon
	\begin{cases}
	v_0^+  \mapsto v_{l(0)}^+ & = v_{0}^+\\
	v_0^-  \mapsto v_{r(0)}^- & = v_1^-\\
	v_1^+  \mapsto v_{l(1)}^+ & = v_0^+\\
	v_1^-  \mapsto v_{r(1)}^- & = v_1^-
	\end{cases}$$
	and so on edges it acts by mapping $[00] \mapsto [10], [01] \mapsto [11], [10] \mapsto [01], [11] \mapsto [01]$ where we use square braces $[ij]$ to distinguish the two-letter word $ij$ from the edge in $\Gamma_\phi$ labeled by that word. As a map on the subcomplex $S_\phi$, $\tilde{\phi}$ acts by permuting edges. In particular, $ER_\phi = S_\phi$. As $S_\phi$ is topologically a circle, which has one connected component, we see that $k = 1$. As $S_\phi$ has four edges and four vertices, we see that $E = V = 4$. This gives $m = E - V + k = 1$. We calculate that $\tilde{H}^0(S_\phi) = \mathbb{Z}^{k-1} = \mathbb{Z}^0 = 0$ and $\check{H}^1(S_\phi) = \mathbb{Z}^m = \mathbb{Z}$. By Corollary \ref{COR:BD}, this tells us that $\check{H}^1(\Omega_\phi) \cong \varinjlim \left[\begin{smallmatrix} 1 & 1 \\ 1 & 1 \end{smallmatrix}\right] \oplus \mathbb{Z}$. By Example \ref{EX:TM-direct-limit}, $$\check{H}^1(\Omega_\phi) \cong \mathbb{Z}[1/2] \oplus \mathbb{Z}.$$
	
	\begin{figure}[H]
\centering

\begin{tikzpicture}[->, node distance=2cm, auto] 
\clip (-5,-3.4) rectangle (5,3.5);
\node [fill,circle,draw,inner sep = 0pt, outer sep = 0pt, minimum size=2mm] (ai) at (2.000000,0.000000) {}; 
\node [fill,circle,draw,inner sep = 0pt, outer sep = 0pt, minimum size=2mm] (ao) at (0.000593,2.000000) {}; 
\node [fill,circle,draw,inner sep = 0pt, outer sep = 0pt, minimum size=2mm] (bi) at (-2.000000,0.001185) {}; 
\node [fill,circle,draw,inner sep = 0pt, outer sep = 0pt, minimum size=2mm] (bo) at (-0.001778,-1.999999) {}; 
\draw (ai) edge[bend right=110, looseness=3, ->, ultra thick, color=red] node {$1$}(ao); 
\draw (bi) edge[bend right=110, looseness=3, ->, ultra thick, color=blue] node {$0$}(bo); 
\draw [-,ultra thick, bend right, draw=white, line width=6pt, looseness=0.7] (bo) to (ai); 
\draw [->,ultra thick, bend right, looseness=0.7, swap] (bo) to node {$10$} (ai); 
\draw [-,ultra thick, bend right, draw=white, line width=6pt, looseness=0.7] (ao) to (bi); 
\draw [->,ultra thick, bend right, looseness=0.7, swap] (ao) to node {$01$} (bi); 
\draw [-,ultra thick, bend right, draw=white, line width=6pt, looseness=0.7] (bo) to (bi); 
\draw [->,ultra thick, bend left, looseness=0.7] (bo) to node {$00$} (bi);
\draw [-,ultra thick, bend right, draw=white, line width=6pt, looseness=0.7] (ao) to (ai); 
\draw [->,ultra thick, bend left, looseness=0.7] (ao) to node {$11$} (ai); 

\end{tikzpicture}
\caption{The Barge-Diamond complex for the Thue-Morse substitution.}
\label{FIG:BD-TM}
\end{figure}
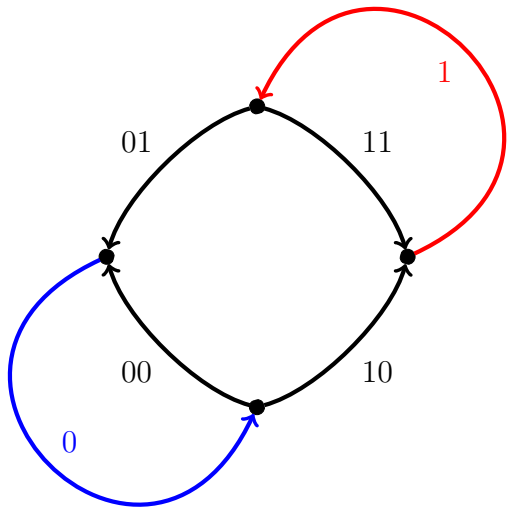

	\end{example}
	
\subsection{Via Anderson-Putnam}\label{SEC:ap-cohomology}
	The second technique for calculating cohomology is via the Anderson-Putnam complex. This method cannot as easily be reduced to the application of combinatorial tricks, such as the Euler characteristic formula and the splitting lemma, in order to circumvent complicated computational problems in homological algebra. We recommend the reader be familiar with simplicial and/or cellular (co)homology $H^i$ when calculating cohomology using this method. In particular, the reader should be comfortable with calculating homomorphisms on cohomology groups induced by continuous maps. A standard text for an introduction to these concepts is Hatcher \cite{H:book}.

	Let $\phi$ be a primitive, recognizable substitution on the alphabet $\mathcal{A}$. Let $K_\phi$ be the Anderson-Putnam complex of $\phi$. Anderson and Putnam have shown that the \Cech cohomology of $\Omega_\phi$ is determined by the direct limit of an induced homomorphism acting on the cohomology of what is now called the Anderson-Putnam complex \cite{AP}. Using the modified AP complex $K_\phi$, G\"{a}hler and Maloney have shown that this complex, as defined in Section \ref{SEC:complexes}, can be used in place of the originally defined AP complex \cite{GM:multi-one-d}. We define the map acting on the AP complex $K_\phi$ in the following way.
	
	Suppose we have an edge $e$ with label $ijk$. Suppose $\phi(i) = i_1 i_2 \cdots i_{|\phi(i)|}$ and similarly for $\phi(j)$ and $\phi(k)$. Define a continuous map called the \emph{collared substitution} $\tilde{\phi} \colon K_\phi \to K_\phi$ by mapping the edge $e$ to the ordered collection of edges with labels
	$$
	[i_{|\phi(i)|} j_1 j_2]
	[j_1 j_2 j_3]
	\cdots
	[j_{|\phi(j)|-2} j_{|\phi(j)|-1} j_{|\phi(j)|}]
	[j_{|\phi(j)|-1} j_{|\phi(j)|} k_1]
	$$
	in an orientation preserving way (the parametrization of the map will not matter).
	
	Let $H^i$ be ordinary simplicial cohomology and, if $f \colon X \to Y$ is a continuous map between topological spaces, let $f^* \colon H^i(Y) \to H^i(X)$ be the induced homomorphism between cohomology groups.
	\begin{example}
	Let $\phi \colon 0 \mapsto 001, 1 \mapsto 01$ be the proper version of the Fibonacci substitution (this will be made formal in Section \ref{SEC:proper}). The word $010$ is admitted by $\phi$ and so there is an edge labelled by $[010]$ in the AP complex $K_\phi$. This edge is mapped by the collared substitution $\tilde{\phi}$ to the collection of edges $[101] [010]$. Similarly, the collared substitution acts on the full list of edges of $K_\phi$ by the rule
	$$
	\begin{array}{rcl}
	[001] & \mapsto & [100] [001] [010]\\
	{[010]} & \mapsto & [101] [010]\\
	{[100]} & \mapsto & [100] [001] [010] \\
	{[101]} & \mapsto & [100] [001] [010].
	\end{array}
	$$
	\end{example}
	The map $\tilde{\phi}$ induces a homomorphism $\tilde{\phi}^\ast \colon H^1(K_\phi) \to H^1(K_\phi)$ on cohomology. We use the following result \cite{GM:multi-one-d} (Anderson and Putnam proved this result in the case that $K_\phi$ is replaced with the original definition of the AP complex \cite{AP}).
	\begin{theorem}\label{THM:AP}
	If $\phi$ is a primitive, recognizable substitution, then there is an isomorphism
		$$\check{H}^1(\Omega_\phi) \cong \varinjlim (H^1(K_\phi), \tilde{\phi}^*).$$
	\end{theorem}
	Let $R = \mbox{rk} H^1(K_\phi)$, the rank of the cohomology of $K_\phi$. Grout finds an explicit generating set of cocycles for the cohomology of $K_\phi$ and then outputs the induced homomorphism $\tilde{\phi}^*$ as the associated $(R \times R)$-matrix $M_{AP}$, which should be interpreted as acting on $\mathbb{Z}^R$ with respect to this generating set. So $\check{H}^1(\Omega_\phi) \cong \varinjlim M_{AP}$.  
	
	The algorithm for this computation begins by constructing the boundary matrix for the AP complex.  To do this, take all of the admitted three-letter words $abc$ and use the convention that the boundary of such an edge is $bc-ab$. Construct the associated $m \times n$ boundary matrix $B$, where $m$ is the number of two-letter words admitted by $\phi$ and $n$ is the number of three-letter words admitted by $\phi$. The entries of $B$ are given by
	$$(B)_{abc,xy} =
	\begin{cases}
	1, & \text{if $xy = bc$;}\\
	-1, & \text{if $xy = ab$;}\\
	0, & \text{otherwise.}
	\end{cases}
	$$
	Using standard methods from linear algebra, we find a maximal set of linearly independent $n$-dimensional integer vectors $g$ such that $Bg = 0$, searching over the set of all vectors of $0$s and $1$s. By construction, this set generates the kernel of the boundary homomorphism inside the simplicial $1$-chain group of $K_\phi$. This gives us a generating set of cycle vectors for the first homology of $K_\phi$.
	
	We then apply the homomorphism induced by the collared substitution on the chain groups to each of these generating vectors, giving us a new set of image vectors.  Using Gaussian elimination, we find the coordinates of these image vectors in terms of the generating vectors.  This induced homomorphism on homology can be represented as a square matrix. The transpose of this matrix $M_{AP}$ then represents the induced homomorphism on cohomology. The matrix $M_{AP}$ is the output for the cohomology calculation via the Anderson-Putnam method.  It should be noted that this algorithm is not efficient in the case where the substitution admits many three-letter words, as the dimension $m$ of the $1$-chain complex is the dominant limiting factor when finding linearly independent generating cycles. The time complexity increases exponentially with respect to $m$.
	\begin{example}\label{EX:TM-AP}
	Let $\phi \colon 0 \mapsto 01, 1 \mapsto 10$ be the Thue-Morse substitution. The admitted two-letter words for $\phi$ are given by $\mathcal{L}^2_\phi =\{00, 01, 10, 11\}$ and  the admitted three-letter words for $\phi$ are given by $\mathcal{L}^3_\phi = \{001, 010, 011, 100, 101, 110\}$. The AP complex $K_\phi$ for $\phi$ is given in Figure \ref{FIG:AP-TM}.
	
	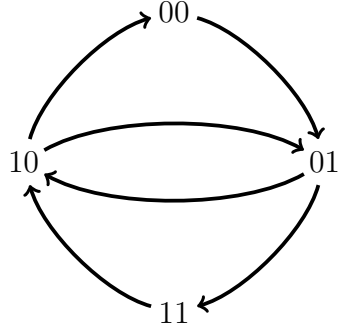
\begin{figure}[H]
\centering
\begin{tikzpicture}[->, node distance=3cm, auto, text=black, line width=0.5mm] 
\clip (-2.5,-2.5) rectangle (2.5,2.5);
\node [circle,inner sep = 0pt, outer sep = 2pt, minimum size=2mm] (abi) at (2.000000,0.000000) {$01$}; 
\node [circle,inner sep = 0pt, outer sep = 2pt, minimum size=2mm] (aai) at (0,2) {$00$};
\node [circle,inner sep = 0pt, outer sep = 2pt, minimum size=2mm] (bai) at (-2,0) {$10$}; 
\node [circle,inner sep = 0pt, outer sep = 2pt, minimum size=2mm] (bbi) at (0,-2) {$11$};

\draw [->, bend left, looseness=0.7] (abi) to (bbi); 

\draw [->, bend left, looseness=0.7] (bai) to (abi); 

\draw [->, bend left, looseness=0.7] (bbi) to (bai);

\draw [->, bend left, looseness=0.7] (abi) to (bai);

\draw [->, bend left, looseness=0.7] (bai) to (aai);

\draw [->, bend left, looseness=0.7] (aai) to (abi);

\end{tikzpicture}
\caption{The AP complex for the Thue-Morse substitution.}\label{FIG:AP-TM}
\end{figure}

The graph $K_\phi$ has Euler characteristic $\chi = |\mathcal{L}^3_\phi| - |\mathcal{L}^2_\phi| = 2$. As $K_\phi$ is connected, it follows that $\operatorname{rk}H^0(K_\phi) = 1$. We deduce that $R = \operatorname{rk}H^1(K_\phi) = \chi + \operatorname{rk}H^0(K_\phi) = 3$.

The boundary matrix $B$ for $K_\phi$ is
$$B =
\left[
\begin{matrix}
-1 & 0 & 0 & 1 & 0 & 0\\
1 & -1 & -1 & 0 & 1 & 0\\
0 & 1 & 0 & -1 & -1 & 1\\
0 & 0 & 1 & 0 & 0 & -1
\end{matrix}
\right].$$
We find three linearly independent vectors $v_1, v_2, v_3$ such that $Bv_i = 0$ by inspection:
$$
v_1 = \left[\begin{matrix}0\\1\\0\\0\\1\\0\end{matrix}\right],
v_2 = \left[\begin{matrix}1\\0\\1\\1\\0\\1\end{matrix}\right],
v_3 = \left[\begin{matrix}0\\0\\1\\0\\1\\1\end{matrix}\right].
$$
The action of substitution on the cycles represented by these vectors is $\tilde{\phi}^*(v_1) = v_2, \tilde{\phi}^*(v_2) = 2v_1 + v_2, \tilde{\phi}^*(v_3) = v_1 + v_2 $ and so the induced homomorphism on homology is represented by the matrix
$$\tilde{M}_{AP} =
\left[
\begin{matrix}
0 & 2 & 1 \\
1 & 1 & 1 \\
0 & 0 & 0
\end{matrix}
\right].$$
On cohomology, the action is given by the transpose of $\tilde{M}_{AP}$. That is, $M_{AP} = \tilde{M}_{AP}^T$. By Theorem \ref{THM:AP}, we calculate cohomology of the tiling space to be $$\check{H}^1(\Omega_\phi) \cong \varinjlim
\left[
\begin{matrix}
0 & 1 & 0 \\
2 & 1 & 0 \\
1 & 1 & 0
\end{matrix}
\right].$$
	\end{example}
	
\subsection{Via properization}\label{SEC:proper}
	For this method of computation we make use of a technique involving the return words that were introduced in Section \ref{SEC:recog}, as outlined in work of Durand, Host and Skau \cite{DHS:return-words}, for replacing a primitive substitution with a conjugate \emph{pre-left-proper} primitive substitution. One may then use the fact that if $\phi$ is a recognizable pre-left-proper primitive substitution, then $\check{H}^1(\Omega_\phi)\cong \varinjlim M_\phi^T$ \cite{S:book}.

	We begin by defining what it means to be proper.

	\begin{definition}\leavevmode
		\item A substitution is \emph{left-proper} if there exists a letter $i \in \mathcal{A}$ such that the leftmost letter of $\phi(j)$ is $i$ for all $j \in \mathcal{A}$. That is, $\phi(j) = i w_j$ for some $w_j \in \mathcal{A}^\ast$.
		\item A substitution is \emph{right-proper} if there exists a letter $i \in \mathcal{A}$ such that the rightmost letter of $\phi(j)$ is $i$ for all $j \in \mathcal{A}$. That is, $\phi(j) = w_j i$ for some $w_j \in \mathcal{A}^\ast$.
		\item A substitution is \emph{fully-proper} (we will often abbreviate this to just \emph{proper}) if it is both left-proper and right-proper.
		\item A substitution is \emph{pre-left-proper} if some power of the substitution is left-proper.
		\item A substitution is \emph{pre-right-proper} if some power of the substitution is right-proper.
		\item A substitution is \emph{pre-proper} if some power of the substitution is proper.
\end{definition}

	The following algorithm produces what we call the \emph{pre-left-properization} of a substitution, so-called because there exists a finite power of the new substitution which is left-proper. As usual, $k$ is the one from the definition of the fixed letter $a$. 

	Note that if $v$ is a return word to the fixed letter $a$, then $va \in \mathcal{L}_\phi$ and so $\phi^k(va)$ must also be admitted by $\phi$. We have $\phi^k(va) = \phi^k(v) \phi^k(a)$, and both of $\phi^k(v)$ and $\phi^k(a)$ begin with the fixed letter $a$, hence $\phi^k(v)$ is an exact concatenation of return words to $a$. So, if we apply $\phi^k$ to a return word, then the result is a unique composition of return words.  We let $\eta$ denote this newly constructed substitution rule on the new alphabet $\mathcal{R}_{\phi,a}$ of return words.

	This process gives a new primitive substitution on a possibly larger alphabet \cite{DHS:return-words}. By taking a power of this substitution, we arrive at a left-proper substitution. It is clear that such a power exists. Indeed, every return word $v \in \mathcal{R}_{\phi,a}$ begins with the fixed letter $a$ and, by primitivity, $\phi^n(a)$ contains at least two copies of the letter $a$ for large enough $n$, so $\phi^n(a) = v_0 a u$ for some return word $v_0 \in \mathcal{R}_{\phi,a}$ and some other word $u$. Then $\eta^n(v)$ must begin with $v_0$. It follows that $\eta^n$ is a left-proper substitution with leftmost letter $v_0$. 

	We may also form an associated proper substitution on $\mathcal{R}_{\phi,a}$ by composing $\eta^n$ with its \emph{right conjugate}. The right conjugate $\phi^{(R)}$ of a left-proper substitution $\phi$ is given by setting $\phi^{(R)}(b) = w_ba$ where $a$ is the fixed letter such that $\phi(b) = aw_b$ for all $b\in\mathcal{A}$. The right conjugate is a right-proper substitution, and the composition of a left-proper and right-proper substitution is both left-proper and right-proper, hence fully-proper.  It is easy to show that $X_{\phi\circ \phi^{(R)}}$ and $X_\phi$ are topologically conjugate subshifts \cite{DL:properisation}. In fact, a word is admitted by $\phi\circ\phi^{(R)}$ if and only if it is admitted by $\phi$, so $X_{\phi\circ \phi^{(R)}}$ and $X_\phi$ are equal. Hence, $\check{H}^1(\Omega_{\phi\circ \phi^{(R)}}) \cong \check{H}^1(\Omega_\phi)$.
	\begin{example}
	The Fibonacci substitution $\phi \colon 0 \mapsto 01 , \mapsto 0$ is left-proper. Its right conjugate is given by $\phi^{(R)} \colon 0 \mapsto 10, 1 \mapsto 0$ which is right-proper. Their composition is given by $\phi \circ \phi^{(R)} \colon 0 \mapsto 001, 1 \mapsto 01$ which is fully-proper.
	\end{example}
	
	We make use of the following which is a consequence of results appearing in the work of Durand, Host and Skau \cite{DHS:return-words}.
	\begin{proposition}
		Let $\phi$ be a primitive substitution on $\mathcal{A}$ and let $\eta$ be the pre-left-properization of $\phi$. The tiling space $\Omega_\eta$ is homeomorphic to $\Omega_\phi$.
	\end{proposition}
	Hence we get a corollary.
	\begin{corollary}\label{COR:pre-left}
	$$\check{H}^1(\Omega_\eta) \cong \check{H}^1(\Omega_\phi).$$
	\end{corollary}
	As $\eta^n$ is left-proper, $\check{H}^1(\Omega_{\eta^n})\cong \varinjlim M_{\eta^n}^T \cong \varinjlim (M_{\eta}^n)^T \cong \varinjlim M_{\eta}^T$. Hence $\check{H}^1(\Omega_\phi)\cong \varinjlim M_\eta^T$.

	Grout outputs the pre-left-properization $\eta$, the left-properization $\eta^n$, the full properization $\eta^n \circ (\eta^n)^{(R)}$, and owing to the above, Grout also outputs the matrix $M_\eta^T$ in the cohomology section.

\begin{figure}[H]
\centerline{
\includegraphics[scale=0.5]{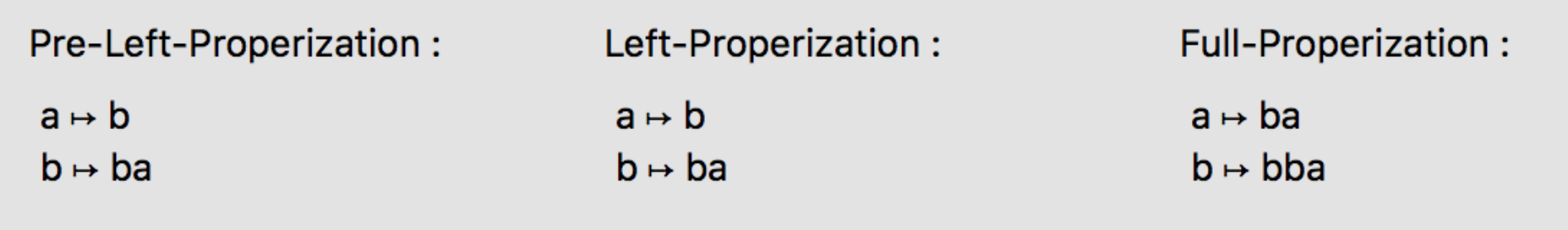}
}
\caption{The properization results for the substitution $a \mapsto b, b\mapsto ba$ in Grout.}
\end{figure}

\begin{remark}
If the substitution is already proper, the properization algorithm may return a different proper version of this substitution.  This may seem like it is a feature which has no use, but by iterating this process, we find that the sequence of substitutions is eventually periodic, first proved by Durand \cite{D:derived-seqs}. This sequence of properizations is itself of interest.
\end{remark}

\begin{figure}[H]
\centerline{
\includegraphics[scale=0.5]{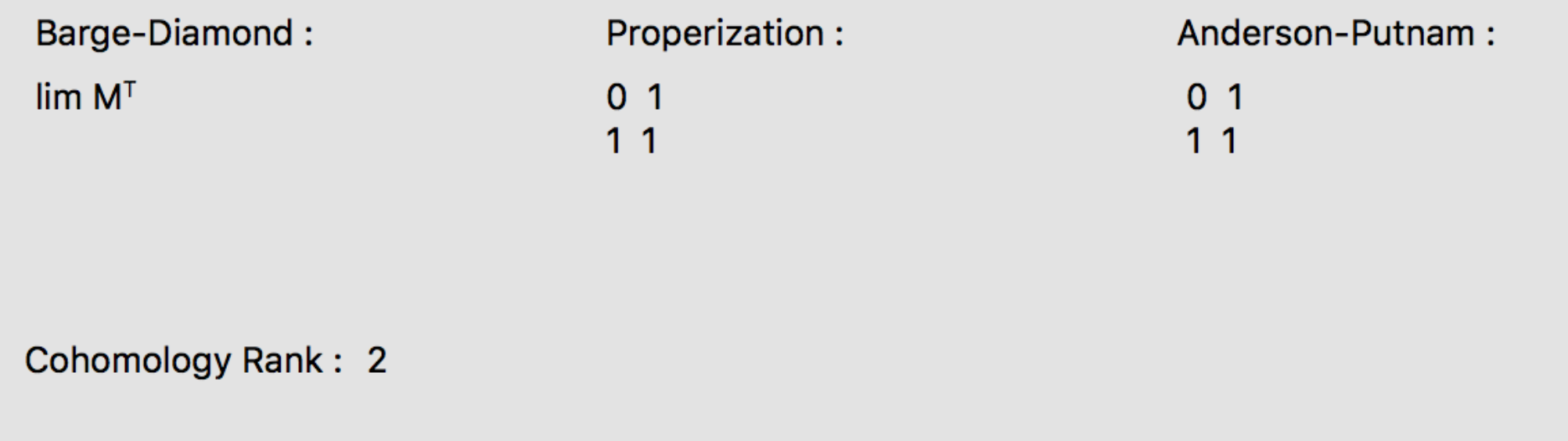}
}
\caption{The cohomology results for the substitution $a \mapsto b, b\mapsto ba$ in Grout.}
\end{figure}
\begin{example}\label{EX:TM-proper}
Let $\phi \colon 0 \mapsto 01, 1\mapsto 10$ be the Thue-Morse substitution. Recall from Example \ref{EX:TM-return} that the Thue-Morse substitution has return words $\mathcal{R}_{\phi,0} = \{0,01,011\}$ which we relabel as $\{[0] = i, [01] = j, [011] = k\}$. The induced substitution on the new alphabet is then given by
$$\eta \colon \left \{\begin{array}{rclclcl}
i & \mapsto & [\phi(0)] & = & [01] & = & j\\
j & \mapsto & [\phi(01)] & = & [0110] & = & ki\\
k & \mapsto & [\phi(011)] & = & [011010] & = & kji
\end{array}\right.,$$
giving us the pre-left-properization. We then calculate cohomology via Corollary \ref{COR:pre-left} to be $$\check{H}^1(\Omega_\phi) \cong \varinjlim M_\eta^T \cong \varinjlim \left[\begin{matrix}0&1&0\\1&0&1\\1&1&1\end{matrix}\right].$$
We invite the reader to check that the square of $\eta$ is a left-proper substitution and that $\eta^2 \circ (\eta^2)^{(R)}$ is a proper substitution.
\end{example}

\section{Pisot substitutions}\label{SEC:pisot}
In this section, we direct our attention to a special class of substitutions known as \emph{Pisot substitutions}. This class of substitutions has received considerable attention in the literature. For an overview of the various flavors of Pisot conjectures and a mostly up-to-date exposition of current knowledge, we refer the reader to the review article of Akiyama, Barge, Berth\'{e} and Lee \cite{ABBL:pisot}.

\begin{definition}
Let $\phi$ be a primitive aperiodic substitution on an alphabet $\mathcal{A}$ with $l$ letters and let $M_\phi$ be the associated substitution matrix with characteristic polynomial $q_{M_\phi}$ and PF-eigenvector $\lambda_{PF}$. We say that $\phi$ is a \emph{Pisot substitution} if for every root $\lambda \neq \lambda_{PF}$ of the minimal polynomial of $\lambda_{PF}$, we have $|\lambda|<1$. If we further have that $p_{M_\phi}$ is an irreducible polynomial over the rationals (equivalently over the integers) then we say that $\phi$ is an \emph{irreducible Pisot substitution}. Equivalently, $\phi$ is irreducible Pisot if every eigenvalue $\lambda \neq \lambda_{PF}$ of $M_\phi$ satisfies $0 < |\lambda| < 1$.
\end{definition}
Note that irreducible Pisot substitutions on two or more letters are necessarily primitive and recognizable.
\begin{example}
Some of the most well-known substitutions are irreducible Pisot:
\begin{itemize}
\item The Fibonacci substitution from Example \ref{EX:fibonacci} has eigenvalues $\lambda_{PF} = \tau$ and $\lambda_1 = \tau - 1$ where $\tau = (1+\sqrt{5})/2$ is the golden mean.

\item The \emph{silver mean substitution} \cite{BG:book} $\phi \colon 0 \mapsto 001, 1\mapsto 0$ with fixed point \seqnum{A171588}, (see also \seqnum{A049472}):
$$001001000100100010010010001 \cdots $$ has eigenvalues $\lambda_{PF} = 1+\sqrt{2}$ and $\lambda_1 = 1-\sqrt{2}$.
\item The \emph{Tribonacci substitution} \cite{R:rauzy-fractals} $\phi \colon 0 \mapsto 01, 1 \mapsto 02, 2 \mapsto 0$ with fixed point \seqnum{A092782}: $$01 02 01 0 01 02 01 01 02 01 0 01 02 01 02 \cdots $$ has eigenvalues $\lambda_{PF} \simeq 1.83929, \lambda_{1,2} \simeq -0.419643 \pm 0.606291 i$.
\item The \emph{flipped Tribonacci substitution} \cite{S:flipped-tribonacci} $\phi \colon 1 \mapsto 12, 2 \mapsto 31, 3 \mapsto 1$ with fixed point \seqnum{A100619}: $$123111212123112311231112123 \cdots$$ has the same eigenvalues as the Tribonacci substitution as they have equal substitution matrices.
\end{itemize}
\end{example}
The Tribonacci substitution and the flipped Tribonacci substitution have very different dynamical properties, even though they appear extremely similar in their presentations and share eigenvalues \cite{S:flipped-tribonacci}.
\begin{example}
The Thue-Morse substitution from Example \ref{EX:thue-morse} is a Pisot substitution as it has PF-eigenvalue $\lambda_{PF} = 2$ whose minimal polynomial is just $\lambda - 2$, and is therefore trivially Pisot. However, this is not an irreducible Pisot substitution because $p_{M_\phi} = \lambda(\lambda - 2)$ is reducible. Equivalently, because $0$ is also an eigenvalue of $M_\phi$, the substitution cannot be irreducible Pisot.
\end{example}
\begin{example}
The substitution $\phi \colon 0 \mapsto 001111, 1 \mapsto 001$ is not Pisot, as the characteristic polynomial of $M_\phi$ is $p_{M_\phi} = \lambda^2 - 3\lambda - 6$, which is irreducible but with both roots of modulus greater than $1$.
\end{example}

An important tool for studying primitive aperiodic symbolic substitutions is the associated dynamical spectrum of both its subshift and tiling space. A good reference text for such notions is the book by Baake and Grimm \cite{BG:book}. The notion of the dynamical spectrum for a compact measure dynamical system is beyond the scope of this work. Fortunately, there is a simple equivalent combinatorial definition of pure discrete spectrum for irreducible Pisot substitutions \cite{BSW:pds}. We adopt this equivalent notion as our definition below in Theorem \ref{THM:pure-discrete}.

\begin{remark}
We note that, due to results of Clark and Sadun \cite{CS:shape}, the subshift of a Pisot substitution has pure discrete spectrum if and only if its corresponding tiling space (with possibly non-unital tile lengths) has pure discrete spectrum. With this in mind, we only mention the dynamical spectrum of the subshift, confident that everything also applies to the tiling space.
\end{remark}

Let $\psi(u) \in \mathbb{Z}^l$ be the abelianization of the word $u$. That is, the $i$th coordinate of the vector $\psi(u)$ is the number of times the letter $i$ appears in the word $u$. We say that the pair of words $(u,v)$ are a \emph{balanced pair} if $\psi(u) = \psi(v)$. If $u$ and $v$ cannot be factored as proper factors $u = u_1u_2$ and $v = v_1v_2$ such that $(u_1,v_1)$ and $(u_2,v_2)$ are balanced pairs, then we say that $(u,v)$ is an \emph{irreducible balanced pair}. For every $i \in \Al$, we call the balanced pair $(i,i)$ a \emph{coincidence}. For pairs $(u_1,v_1)$ and $(u_2,v_2)$ we write $(u_1,v_1)(u_2,v_2) = (u_1u_2,v_1v_2)$ and call this \emph{concatenation of pairs}. We say that $(u_1,v_1)$ and $(u_2,v_2)$ are \emph{factors} of $(u_1u_2,v_1v_2)$. Clearly, every balanced pair can be factored uniquely into irreducible balanced pairs and if $(u,v)$ is a balanced pair, then $(\phi(u),\phi(v))$ is a balanced pair.

For a balanced pair $(u,v)$, let
$$
\begin{array}{rcl}
I(u,v) 	& = & \{(u_i,v_i) \mid \exists n\geq 0, \mbox{ such that }(u_i,v_i) \mbox{ appears as}\\
		&	& \mbox{an irreducible balanced pair factor of } (\phi^n(u),\phi^n(v))\}.
\end{array}
$$

\begin{definition}
Let $\phi$ be a primitive aperiodic substitution on an alphabet $\mathcal{A}$ with $l$ letters. If $I(u,v)$ is finite, we say that the balanced pair algorithm for $\phi$ \emph{terminates} on $(u,v)$. If $I(u,v)$ contains a coincidence, then we say that the balanced pair algorithm for $\phi$ \emph{terminates with coincidence} on $(u,v)$.
\end{definition}
\begin{example}\label{EX:Fib-coin}
Let $\phi \colon 0 \mapsto 01, 1 \mapsto 0$ be the Fibonacci substitution. Consider the balanced pair $(01,10)$. We find that $\left(\phi(01),\phi(10)\right) = (010,001) = (0,0)(10,01)$ and the next iterate also contains the coincidence $(1,1)$. It is clear that no new irreducible balanced pairs are found as factors of higher iterates of the substitution. It follows that $I(01,10)$ is given by $$I(01,10) = \{(0,0),(1,1),(01,10),(10,01)\},$$ which is both finite and contains a coincidence. Hence, $\phi$ terminates with coincidence on $(01,10)$.
\end{example}
The following theorem \cite{BSW:pds} gives a simple criterion in terms of the balanced pair algorithm for determining if the subshift of an irreducible Pisot substitution has pure discrete spectrum. We take this as our definition for pure discrete spectrum.
\begin{theorem}\label{THM:pure-discrete}
If $\phi$ is an irreducible Pisot substitution, then for any distinct pair of letters $i,j \in \Al$ the subshift $X_\phi$ has pure discrete spectrum if and only if $\phi$ terminates with coincidence on the balanced pair $(ij,ji)$.
\end{theorem}
 From Example \ref{EX:Fib-coin}, the above theorem tells us that the Fibonacci substitution has pure discrete spectrum.

The following notion of strong coincidence is different from the balanced pair algorithm terminating with coincidence, though they are intimately related (and possibly equivalent) \cite{BD:pisot-coincidence}.
\begin{definition}
Let $\phi$ be a primitive aperiodic substitution on an alphabet $\mathcal{A}$ with $l$ letters. For a distinct pair of letters $i,j \in \mathcal{A}$, if there exists $n \geq 1$ and a letter $k \in \mathcal{A}$ such that $\phi^n(i) = ukv$, $\phi^n(j) = \tilde{u}k\tilde{v}$ and $\psi(u) = \psi(\tilde{u})$ then we say that $\phi$ has a \emph{coincidence} for the pair $i,j$. If $\phi$ has a coincidence for every pair of letters $i,j \in \mathcal{A}$ then we say that $\phi$ is \emph{strongly coincident}.
\end{definition}
\begin{example}
Let $\phi \colon 0 \mapsto 10, 1 \mapsto 0$ be the reversed Fibonacci substitution. Note that $\phi^3(0) = (10)(0)(10)$ and $\phi^3(1) = (01)(0)()$, where we have used brackets to partition the substituted words appropriately. So, $\phi$ has a coincidence for the pair $0,1$ (with $\psi(u) = \psi(\tilde{u}) = (1,1)$, $k=0$). As $\mathcal{A} = \{0,1\}$, we have checked all distinct pairs of letters. It follows that the reversed Fibonacci substitution is strongly coincident.
\end{example}
Note that, unlike in the above example, the ordinary Fibonacci substitution reaches a coincidence for the pair $0,1$ after just the first iteration. This is an artifact of the asymmetric definition for a substitution to have a coincidence.

Solomyak has shown that if the subshift of a primitive substitution has pure discrete spectrum, then the substitution must be Pisot \cite{S:pisot}. We are concerned with the converse of this statement in the case of irreducible Pisot substitutions and a close variant concerning strong coincidence\footnote{Although the first of these conjectures is referred to just by the name \emph{Pisot}, the conjecture is sometimes referred to as the \emph{Pisot substitution conjecture} to distinguish it from conjectures in other areas adopting the same name.}.
\begin{conj}[Pisot conjecture]
The subshift of every irreducible Pisot substitution has pure discrete spectrum.
\end{conj}

\begin{conj}[Strong coincidence conjecture]\label{scoin}
Every irreducible Pisot substitution is strongly coincident.
\end{conj}

It is not known, in general, if the subshift of every strongly coincident substitution has pure discrete spectrum. If this is the case, then the Pisot conjecture is equivalent to the strong coincidence conjecture.

Other than in specific cases of single substitutions, there are relatively few families of substitutions for which we know that the Pisot conjecture holds:
Hollander and Solomyak have shown that for two-letter Pisot substitutions, the strong coincidence conjecture is equivalent to the Pisot conjecture \cite{HS:pisot};
Barge and Diamond had previously shown that the strong coincidence conjecture is true for two-letter substitutions and hence that the Pisot conjecture holds for all substitutions on two letters \cite{BD:pisot-coincidence};
Barge has shown that the Pisot conjecture holds for substitutions of $\beta$-type \cite{B:pisot-beta-subs}; Akiyama, G\"{a}hler and Lee have shown via an exhaustive search that the Pisot conjecture holds for all three-letter substitutions whose substitution matrices have trace at most 2 \cite{AGL:pisot};
Berth\'{e}, Jolivet, and Siegel have shown that the Pisot conjecture holds for substitutions of Arnoux-Rauzy type \cite{BJS:arnoux-rauzy};
Barge has shown that the Pisot conjecture holds for all substitutions that are left-proper and right-bijective (or vice-versa) \cite{B:proper-bijective};
Barge, Bruin, Jones and Sadun have shown that a na\"{i}ve homological analogue of the Pisot conjecture admits counterexamples for all algebraic degrees, but that the coincidence rank conjecture (a variant of the homological Pisot conjecture relating the norm of the PF-eigenvalue and the \emph{coincidence rank} of the substitution) holds for any substitution whose PF-eigenvalue has algebraic degree 1 \cite{BBJS:homological-pisot}.

\subsection{Computational search for counterexamples}
Perhaps the largest computer search for counterexamples to either conjecture in the literature appears in work of Akiyama, G\"{a}hler and Lee \cite{AGL:pisot}, where all 446,683 irreducible Pisot substitutions on three letters whose substitution matrix has trace at most 2 were checked for counterexamples to the Pisot conjecture.  Instead, we choose to focus our search for counterexamples to the strong coincidence conjecture. We use the algorithms developed for Grout, as well as standard algorithms from linear algebra, to calculate eigenvalues and check for irreducibility over $\mathbb{Z}$ \cite{L:factorising}, in order to do a large scale search for counterexamples to the strong coincidence conjecture. For our search, we limit to substitutions over three or four letters, where irreducibility can be checked in a combinatorial fashion instead of full implementation of the irreducibility algorithms.

We would like to highlight a subtlety encountered while implementing the check that involved a deterministic check for when eigenvalues have modulus exactly 1 or are instead very close but strictly within or without the unit disc. Algorithms for finding eigenvalues are typically numeric and involve a converging sequence of approximate solutions to the true solution. In order for the process to halt in finite time, an acceptance threshold needs to be set. This is problematic when one of the conditions that needs to be checked is that an eigenvalue does not lie on the boundary of the unit disc. This problem was overcome by implementing an algorithm that was brought to our attention on Stack Exchange \cite{R:stack-exchange}. Once it is determined that an eigenvalue does not lie on the boundary of the unit disc, the numerical process can be sure to satisfy a threshold around the approximate solutions which is either fully within or fully outside the unit disc, and so the algorithm halts with the correct evaluation of $|\lambda| < 1$, $|\lambda| = 1$, or $|\lambda| > 1$.

We place an ordering on substitutions over the same alphabet (modulo the trivial equivalences given by reversing and letter permutations) via the lexicographical ordering of the words $(\phi(1), \phi(2) ,\ldots, \phi(l))$. For example, $(0,0,0) \prec (0,0,1) \prec (0,2,0) \prec (00,0,0)$. We then methodically generate the substitutions in this list. Sample sizes were limited by allotted computation times; 24 hours for each run on a high performance computing cluster at the University of Leicester, UK.
\begin{res}[Three letters]
The first 321,425,442 three-letter substitutions were checked, of which 12,573,955 were irreducible Pisot.  All of these substitutions passed the strong coincidence check, and therefore no counterexample to the strong coincidence conjecture was found.  During these checks, the number of iterations of the substitution that it took for a strong coincidence to be found was recorded. This is summarized in Table \ref{Tab:strong-coincidence}.

\begin{table}[H]
\centering

\begin{tabular}{c|c}

\textbf{Number of Iterations} & \textbf{Number of Substitutions} \\ \hline

1                             & 1,921,144                        \\ 
2                             & 5,778,331                        \\ 
3                             & 3,777,324                        \\ 
4                             & 984,669                          \\
5                             & 105,325                          \\
6                             & 6,608                            \\
7                             & 512                              \\
8                             & 37                               \\
9                             & 3                                \\
10                            & 2                                \\
11+                           & 0                                \\
\end{tabular}
\caption{Number of iterations to produce a strong coincidence.}\label{Tab:strong-coincidence}

\end{table}
\end{res}
The substitutions that took 10 iterations to form their first coincidence for all pairs of letters were
$$
\begin{array}{rlrlrl}
0 \mapsto & 2011   & 1 \mapsto & 02 & 2 \mapsto & 0, \\
0 \mapsto & 212101 & 1 \mapsto & 0  & 2 \mapsto & 1.
\end{array}
$$
The result suggests a likelihood that there is no bound on the number of iterations necessary for a strongly coincident substitution to form its first coincidence for all pairs, although a proof of this statement is not known to us. Unfortunately, sampling the substitutions formed which produced a new `record' as the program ran did not provide any clear pattern that one might use to produce an infinite family of substitutions who each require a successively larger number of iterations.
\begin{res}[Four letters]
The first 200,399 four-letter substitutions were checked, of which 15,001 were irreducible Pisot. All of these substitutions passed the strong coincidence check, and therefore no counterexample to the strong coincidence conjecture was found.
\end{res}

\section{Supplementary resources}
Grout is available to download for Windows and Mac OSX along with the supporting documentation at the following URL:

\vspace{3mm}

\centerline{\url{https://github.com/GroutGUI/Grout}}

\section{Acknowledgments}
This research used the SPECTRE High Performance Computing Facility at the University of Leicester.

The authors would like to thank Fabien Durand for his helpful advice in piecing together a robust recognizability algorithm, and Etienne Pillin for helpful comments with regards to the coding.  We would also like to thank Alex Clark, Greg Maloney and Jamie Walton for helpful suggestions during the development of Grout and for testing early versions of the program, and Marcy Barge for suggestions and discussions relating to the section on Pisot substitutions.

\bibliographystyle{jis}
\bibliography{bib-grout}

\begin{thebibliography}{10}

\bibitem{ABBL:pisot}
S.~Akiyama, M.~Barge, V.~Berth{\'e}, J.-Y. Lee, and A.~Siegel, On the {P}isot
  substitution conjecture.
\newblock In {\em Mathematics of Aperiodic Order}, Vol. 309 of {\em Progr.
  Math.}, pp.  33--72. Birkh\"auser/Springer, 2015.

\bibitem{AGL:pisot}
S.~Akiyama, F.~G{\"a}hler, and J.-Y. Lee, Determining pure discrete spectrum
  for some self-affine tilings, {\em Discrete Math. Theor. Comput. Sci.} {\bf
  16} (2014), 305--316.

\bibitem{AP}
J.~E. Anderson and I.~F. Putnam, Topological invariants for substitution
  tilings and their associated {$C^*$}-algebras, {\em Ergodic Theory Dynam.
  Systems} {\bf 18} (1998), 509--537.

\bibitem{BGG:substitutions}
M.~Baake, F.~G{\"a}hler, and U.~Grimm, Examples of substitution systems and
  their factors, {\em J. Integer Seq.} {\bf 16} (2013), Article 13.2.14, 18.

\bibitem{BG:book}
M.~Baake and U.~Grimm, {\em Aperiodic Order. Volume 1: A Mathematical
  Invitation}, Vol. 149 of {\em Encyclopedia Math. Appl.}, Cambridge Univ.
  Press, 2013.

\bibitem{B:proper-bijective}
M.~Barge, Pure discrete spectrum for a class of one-dimensional substitution
  tiling systems, {\em Discrete Contin. Dyn. Syst.} {\bf 36} (2016),
  1159--1173.

\bibitem{B:pisot-beta-subs}
M.~{Barge}, The {P}isot conjecture for $\beta$-substitutions, {\em to appear in
  Ergodic Theory Dynam. Systems}  (2017).

\bibitem{BBJS:homological-pisot}
M.~Barge, H.~Bruin, L.~Jones, and L.~Sadun, Homological {P}isot substitutions
  and exact regularity, {\em Israel J. Math.} {\bf 188} (2012), 281--300.

\bibitem{BD:pisot-coincidence}
M.~Barge and B.~Diamond, Coincidence for substitutions of {P}isot type, {\em
  Bull. Soc. Math. France} {\bf 130} (2002), 619--626.

\bibitem{BD}
M.~Barge and B.~Diamond, Cohomology in one-dimensional substitution tiling
  spaces, {\em Proc. Amer. Math. Soc.} {\bf 136} (2008), 2183--2191.

\bibitem{BSW:pds}
M.~Barge, S.~{\v{S}}timac, and R.~F. Williams, Pure discrete spectrum in
  substitution tiling spaces, {\em Discrete Contin. Dyn. Syst.} {\bf 33}
  (2013), 579--597.

\bibitem{BHZ:gap-labeling}
J.~Bellissard, D.~J.~L. Herrmann, and M.~Zarrouati, Hulls of aperiodic solids
  and gap labeling theorems.
\newblock In {\em Directions in Mathematical Quasicrystals}, Vol.~13 of {\em
  CRM Monogr. Ser.}, pp.  207--258. Amer. Math. Soc., 2000.

\bibitem{B:undecidability}
R.~Berger, The undecidability of the domino problem, {\em Mem. Amer. Math.
  Soc.} {\bf 66} (1966), 1--72.

\bibitem{BJS:arnoux-rauzy}
V.~Berth{\'e}, T.~Jolivet, and A.~Siegel, Substitutive {A}rnoux-{R}auzy
  sequences have pure discrete spectrum, {\em Unif. Distrib. Theory} {\bf 7}
  (2012), 173--197.

\bibitem{BS:beta-numerations}
V.~Berth{\'e} and A.~Siegel, Tilings associated with beta-numeration and
  substitutions, {\em Integers} {\bf 5} (2005), A2, 1--46.

\bibitem{B:bott-tu}
R.~Bott and L.~W. Tu, {\em Differential Forms in Algebraic Topology}, Grad.
  Texts in Math., Springer, 1982.

\bibitem{R:stack-exchange}
G.~Bourgeois, How to test if a primitive matrix has an eigenvalue of unit
  modulus, \url{http://math.stackexchange.com/q/1440992}, Math. Stack Exchange,
  2015.

\bibitem{BH:shift-equivalence}
M.~Boyle and D.~Handelman, Algebraic shift equivalence and primitive matrices,
  {\em Trans. Amer. Math. Soc.} {\bf 336} (1993), 121--149.

\bibitem{CE:book}
H.~Cartan and S.~Eilenberg, {\em Homological Algebra}, Princeton Univ. Press,
  1956.

\bibitem{C:chacon-sub}
R.~V. Chacon, Weakly mixing transformations which are not strongly mixing, {\em
  Proc. Amer. Math. Soc.} {\bf 22} (1969), 559--562.

\bibitem{CH:codim-one-attractors}
A.~Clark and J.~Hunton, Tiling spaces, codimension one attractors and shape,
  {\em New York J. Math.} {\bf 18} (2012), 765--796.

\bibitem{CS:shape}
A.~Clark and L.~Sadun, When shape matters: deformations of tiling spaces, {\em
  Ergodic Theory Dynam. Systems} {\bf 26} (2006), 69--86.

\bibitem{C:iterates-of-words}
K.~Culik, II, The decidability of {$\upsilon $}-local catenativity and of other
  properties of {${\rm D}0{\rm L}$} systems, {\em Inform. Process. Lett.} {\bf
  7} (1978), 33--35.

\bibitem{D:period-doubling}
D.~Damanik, Local symmetries in the period-doubling sequence, {\em Discrete
  Appl. Math.} {\bf 100} (2000), 115--121.

\bibitem{D:decomposable}
M.~Dugas, Torsion-free abelian groups defined by an integral matrix, {\em Int.
  J. Algebra} {\bf 6} (2012), 85--99.

\bibitem{D:derived-seqs}
F.~Durand, A characterization of substitutive sequences using return words,
  {\em Discrete Math.} {\bf 179} (1998), 89--101.

\bibitem{DHS:return-words}
F.~Durand, B.~Host, and C.~Skau, Substitutional dynamical systems, {B}ratteli
  diagrams and dimension groups, {\em Ergodic Theory Dynam. Systems} {\bf 19}
  (1999), 953--993.

\bibitem{DL:properisation}
F.~Durand and J.~Leroy, {$S$}-adic conjecture and {B}ratteli diagrams, {\em C.
  R. Math. Acad. Sci. Paris} {\bf 350} (2012), 979--983.

\bibitem{ER:iterates-of-words}
A.~Ehrenfeucht and G.~Rozenberg, On simplifications of {PDOL} systems, In {\em
  Proceedings of a {C}onference on {T}heoretical {C}omputer {S}cience ({U}niv.
  {W}aterloo, {W}aterloo, {O}nt., 1977)}, pp.  81--87. Comput. Sci. Dept.,
  Univ. Waterloo, Waterloo, Ont., 1978.

\bibitem{F:chacon}
S.~Ferenczi, Les transformations de {C}hacon: combinatoire, structure
  g\'eom\'etrique, lien avec les syst\`emes de complexit\'e {$2n+1$}, {\em
  Bull. Soc. Math. France} {\bf 123} (1995), 271--292.

\bibitem{F:complexity}
S.~Ferenczi, Rank and symbolic complexity, {\em Ergodic Theory Dynam. Systems}
  {\bf 16} (1996), 663--682.

\bibitem{F:complexity2}
S.~Ferenczi, Complexity of sequences and dynamical systems, {\em Discrete
  Math.} {\bf 206} (1999), 145--154.
\newblock Combinatorics and number theory (Tiruchirappalli, 1996).

\bibitem{F:book}
N.~P. Fogg, {\em Substitutions in Dynamics, Arithmetics and Combinatorics},
  Vol. 1794 of {\em Lecture Notes in Math.}, Springer-Verlag, 2002.

\bibitem{GM:multi-one-d}
F.~G{\"a}hler and G.~R. Maloney, Cohomology of one-dimensional mixed
  substitution tiling spaces, {\em Topology Appl.} {\bf 160} (2013), 703--719.

\bibitem{GSS:subword-complexity}
D.~Go\v{c}, L.~Schaeffer, and J.~Shallit, Subword complexity and
  {$k$}-synchronization.
\newblock In {\em Developments in Language Theory}, Vol. 7907 of {\em Lecture
  Notes in Comput. Sci.}, pp.  252--263. Springer, 2013.

\bibitem{HL:periodicity}
T.~Harju and M.~Linna, On the periodicity of morphisms on free monoids, {\em
  RAIRO Theor. Inform. Appl.} {\bf 20} (1986), 47--54.

\bibitem{H:book}
A.~Hatcher, {\em Algebraic Topology}, Cambridge Univ. Press, 2002.

\bibitem{HS:pisot}
M.~Hollander and B.~Solomyak, Two-symbol {P}isot substitutions have pure
  discrete spectrum, {\em Ergodic Theory Dynam. Systems} {\bf 23} (2003),
  533--540.

\bibitem{J:complexity}
A.~Julien, Complexity and cohomology for cut-and-projection tilings, {\em
  Ergodic Theory Dynam. Systems} {\bf 30} (2010), 489--523.

\bibitem{L:comp-sci-book}
A.~Lagae, {\em Wang Tiles in Computer Graphics}, Synthesis Lectures on Computer
  Graphics and Animation, Morgan \& Claypool Publishers, 2009.

\bibitem{L:factorising}
A.~K. Lenstra, H.~W. Lenstra, Jr., and L.~Lov{\'a}sz, Factoring polynomials
  with rational coefficients, {\em Math. Ann.} {\bf 261} (1982), 515--534.

\bibitem{LM:introduction-to-symbolic}
D.~Lind and B.~Marcus, {\em An Introduction to Symbolic Dynamics and Coding},
  Cambridge Univ. Press, 1995.

\bibitem{MR:non-primitive}
G.~R. Maloney and D.~Rust, Beyond primitivity for one-dimensional substitution
  subshifts and tiling spaces, {\em to appear in Ergodic Theory Dynam. Systems}
   (2017).

\bibitem{MH:complexity}
M.~Morse and G.~A. Hedlund, Symbolic dynamics, {\em Amer. J. Math.} {\bf 60}
  (1938), 815--866.

\bibitem{M:aperiodic}
B.~Moss{\'e}, Puissances de mots et reconnaissabilit\'e des points fixes d'une
  substitution, {\em Theoret. Comput. Sci.} {\bf 99} (1992), 327--334.

\bibitem{P:pentaplexity}
R.~Penrose, Pentaplexity: A class of non-periodic tilings of the plane, {\em
  Math. Intelligencer} {\bf 2} (1979), 32--37.

\bibitem{P:thue-morse}
E.~Prouhet, M\'emoir sur quelques relations entre les puissances des nombres,
  {\em C. R. Math. Acad. Sci. Paris} {\bf 1} (1851), 225.

\bibitem{R:rauzy-fractals}
G.~Rauzy, Nombres alg\'ebriques et substitutions, {\em Bull. Soc. Math. France}
  {\bf 110} (1982), 147--178.

\bibitem{R:rauzy-graphs}
G.~Rauzy, Suites \`a termes dans un alphabet fini.
\newblock In {\em Seminar on Number Theory, 1982--1983 ({T}alence, 1982/1983)},
  pp.  Exp. No. 25, 16. Univ. Bordeaux I, Talence, 1983.

\bibitem{R:rudin-shapiro}
W.~Rudin, Some theorems on {F}ourier coefficients, {\em Proc. Amer. Math. Soc.}
  {\bf 10} (1959), 855--859.

\bibitem{S:frequency}
K.~Saari, On the frequency of letters in morphic sequences.
\newblock In {\em Computer Science---Theory and Applications}, Vol. 3967 of
  {\em Lecture Notes in Comput. Sci.}, pp.  334--345. Springer, 2006.

\bibitem{S:book}
L.~Sadun, {\em Topology of Tiling Spaces}, Vol.~46 of {\em Univ. Lecture Ser.},
  Amer. Math. Soc., 2008.

\bibitem{SW:bi-infinite}
J.~Shallit and M.~Wang, On two-sided infinite fixed points of morphisms, {\em
  Theoret. Comput. Sci.} {\bf 270} (2002), 659--675.

\bibitem{S:nobel}
D.~Shechtman, I.~Blech, D.~Gratias, and J.~W. Cahn, Metallic phase with
  long-range orientational order and no translational symmetry, {\em Phys. Rev.
  Lett.} {\bf 53} (1984), 1951--1953.

\bibitem{S:flipped-tribonacci}
V.~F. Sirvent, Semigroups and the self-similar structure of the flipped
  {T}ribonacci substitution, {\em Appl. Math. Lett.} {\bf 12} (1999), 25--29.

\bibitem{OEIS}
N.~J.~A. Sloane, {\em The On-Line Encyclopedia of Integer Sequences}.
\newblock Published electronically at \url{http://oeis.org}.

\bibitem{S:pisot}
B.~Solomyak, Dynamics of self-similar tilings, {\em Ergodic Theory Dynam.
  Systems} {\bf 17} (1997), 695--738.

\bibitem{S:aperiodic}
B.~Solomyak, Nonperiodicity implies unique composition for self-similar
  translationally finite tilings, {\em Discrete Comput. Geom.} {\bf 20} (1998),
  265--279.

\bibitem{T:thue-morse}
A.~Thue, \"{U}ber unendliche {Z}eichenreihen, {\em Norske vid. Selsk. Skr. Mat.
  Nat. Kl.} {\bf 7} (1906), 1--22.

\bibitem{T:virology-tilings}
R.~Twarock, A tiling approach to virus capsid assembly explaining a structural
  puzzle in virology, {\em J. Theoret. Biol.} {\bf 226} (2004), 477--482.

\end{thebibliography}
\bigskip
\hrule
\bigskip

\noindent 2010 {\it Mathematics Subject Classification}: Primary 37B10; Secondary 52C23, 54H20, 55N05, 68R15.

\noindent \emph{Keywords:} Tiling spaces, Symbolic dynamics, \Cech cohomology, Pisot, Substitutions.

\bigskip
\hrule
\bigskip

\noindent
(Concerned with sequences
\seqnum{A001285},
\seqnum{A003842},
\seqnum{A003849},
\seqnum{A010059},
\seqnum{A010060},
\seqnum{A014577},
\seqnum{A014709},
\seqnum{A014710},
\seqnum{A020985},
\seqnum{A035263},
\seqnum{A049320},
\seqnum{A049321},
\seqnum{A049472},
\seqnum{A073057},
\seqnum{A092782},
\seqnum{A096268},
\seqnum{A096270},
\seqnum{A100260},
\seqnum{A100619},
\seqnum{A106665},
\seqnum{A171588},
\seqnum{A275202}, and
\seqnum{A275855}.)

\bigskip
\hrule
\bigskip

\end{document}